\documentclass{amsart}

\usepackage[unicode=true,bookmarks=true,bookmarksnumbered=false,bookmarksopen=false, breaklinks=false,pdfborder={0 0 1},backref=false,colorlinks=false] {hyperref}
\hypersetup{pdfauthor={Daniel Holmes, Giosu{\`e} Muratore}}
\usepackage{amssymb, amsmath}
\usepackage{comment}
\usepackage{enumitem}
\usepackage{soul}
\usepackage[all]{xy}
\usepackage{tikz}
\usetikzlibrary{positioning}
\tikzset{every picture/.style={line width=0.75pt}} %set default line width to 0.75pt  
\usepackage{fancyvrb}

\DefineVerbatimEnvironment{MyVerb}{Verbatim}{commandchars=\\\{\}}
\usepackage{framed}
\definecolor{shadecolor}{rgb}{0.95,0.95,0.95}

\numberwithin{equation}{section}
\numberwithin{figure}{section}
\numberwithin{table}{section}
\theoremstyle{plain} %Theorem, Lemma, Corollary, Proposition, Conjecture, Criterion, Assertion
\newtheorem{thm}{\protect\theoremname}[section]
\newtheorem{prop}[thm]{\protect\propositionname}
\newtheorem{cor}[thm]{\protect\corollaryname}
\newtheorem{lem}[thm]{\protect\lemmaname}

\theoremstyle{definition} %Definition, Condition, Problem, Example, Exercise, Algorithm, Question, Axiom, Property, Assumption, Hypothesis
\newtheorem{defn}[thm]{\protect\definitionname}
\newtheorem{example}[thm]{\protect\examplename}

\newtheorem{question}[thm]{\protect\questionname}

\theoremstyle{remark} %Remark, Note, Notation, Claim, Summary, Acknowledgment, Case, Conclusion
\newtheorem{rem}[thm]{\protect\remarkname}
\newtheorem*{ack}{\protect\acknowledgmentkname}

\numberwithin{equation}{section}
\newcommand{\giosue}[1]{{{\color{red}#1}}}

\makeatother

\providecommand{\corollaryname}{Corollary}
\providecommand{\definitionname}{Definition}
\providecommand{\examplename}{Example}
\providecommand{\propositionname}{Proposition}
\providecommand{\remarkname}{Remark}
\providecommand{\theoremname}{Theorem}
\providecommand{\lemmaname}{Lemma}
\providecommand{\problemname}{Problem}
\providecommand{\acknowledgmentkname}{Acknowledgements}
\providecommand{\claimname}{Claim}
\providecommand{\conjecturename}{Conjecture}
\providecommand{\questionname}{Question}

\newcommand{\pkg}{\texttt{GKMtools.jl}}
\newcommand{\pkgwebsite}{\href{https://mgemath.github.io/GKMtools.jl/stable/}{https://mgemath.github.io/GKMtools.jl/}}
\newcommand{\pkgversion}{$0.14.2$}

\newcommand{\ev}{\mathrm{ev}}
\begin{document}

% Short title must have at most 50 characters including spaces
\title[Computations in Equivar. GW Theory of GKM spaces]{Computations in Equivariant Gromov--Witten theory of GKM spaces}


\author[Holmes]{Daniel Holmes}
\address{ISTA (Institute of Science and Technology Austria)}
\email{daniel.holmes@ist.ac.at}
\urladdr{https://www.daniel-holmes.at}

\author[Muratore]{Giosu{\`e} Muratore}
\address{CEMS.UL (University of Lisbon), and  
COPELABS/DEISI (Lus\'ofona University)}
\email{gemuratore@fc.ul.pt}
\urladdr{sites.google.com/view/giosue-muratore}

% \curraddr{Department of Mathematics and Statistics,
% Case Western Reserve University, Cleveland, Ohio 43403}
% \thanks{}

%    General info

% \dedicatory{This paper is dedicated to our advisors.}


\begin{abstract}
We study equivariant Gromov--Witten invariants and quantum cohomology in GKM theory. Building on the localization formula, we prove that the resulting expression is independent of the choice of compatible connection, and provide an equivalent formulation without auxiliary choices. Motivated by this theoretical refinement, we develop a software package, 
%\texttt{GKMtools}.\texttt{jl}, 
$\mathtt{GKMtools . jl}$, 
that implements the computation of equivariant GW invariants and quantum products directly from the GKM graph. We apply our framework to several geometric settings: Calabi--Yau rank two vector bundles on the projective line, where we obtain a new proof for a recent connection to Donaldson--Thomas theory of Kronecker quivers; twisted flag manifolds, which give symplectic but non-algebraic examples of GKM spaces; realizability questions for abstract GKM graphs; classical enumerative problems involving curves in hyperplanes; and quantum Schubert calculus for smooth Schubert varieties. These results demonstrate both the theoretical flexibility of GKM methods and the effectiveness of computational tools in exploring new phenomena.
\end{abstract}

\maketitle
    
\section{Introduction}
GKM theory (after Goresky--Kottwitz--MacPherson \cite{GKM98}) provides a combinatorial framework to study the geometry and topology of certain $T$-spaces with finitely many fixed points.
The essential idea is that the equivariant cohomology ring can be encoded in a finite graph, the \emph{GKM graph}, together with edge-label data arising from the isotropy representations. This correspondence has become a valuable tool in equivariant symplectic geometry, algebraic geometry, and representation theory, as it translates global geometric phenomena into tractable combinatorial structures.

Beyond cohomology, GKM theory has also found applications in enumerative geometry, in particular in the study of equivariant Gromov--Witten invariants and quantum cohomology. 
Localization techniques reduce the evaluation of these invariants to contributions associated to fixed loci of the torus action on the moduli space, which can themselves be expressed in terms of GKM graphs. This approach has led to explicit formulas for Gromov--Witten invariants of GKM spaces \cite{LS17, Hirschi_2023}. Nevertheless, the resulting expressions are often complicated by auxiliary data, most notably the requirement to identify the compatible connection on the GKM graph induced by the geometry of the space, and by the intricate combinatorics of the resulting contributions.

The purpose of this paper is twofold. On the theoretical side, we revisit the localization formula for equivariant Gromov--Witten invariants in the GKM setting, and establish a reformulation that is independent of any choice of connection. On the computational side, we introduce a software package, {\pkg}, designed to make the theoretical machinery of GKM theory effective in explicit calculations. The package provides a general framework for working with GKM graphs, for computing equivariant Gromov--Witten invariants, and for evaluating quantum products.

Let us highlight three of our mathematical results.

\begin{thm}[Theorem \ref{cor:GKM_determines_GW}]
    The equivariant Gromov–Witten invariants of an algebraic or Hamiltonian GKM space can be computed purely in terms of the GKM graph, without
    knowledge of the compatible connection induced by the space.
\end{thm}

\begin{thm}[Theorem \ref{thm:gw_equivariant_CY}, \S\ref{sec:BPS}]
\label{thm:intro_GW_DT_Quiver}
    Let $X_k$ be the GKM variety given by the total space of $\mathcal{O}_{\mathbb{P}^1}(k-1)\oplus \mathcal{O}_{\mathbb{P}^1}(-k-1)$, linearized by an algebraic torus $T$ acting with exactly two fixed points. Suppose $k\ge 1$.
    If $X_k$ is \emph{equivariantly Calabi--Yau}, that is $c_1^T(T_{X_k})=0$, then
    \begin{equation}\label{eq:intro_GW}
    GW_{0,0}^{X_k,d\mathbf{0}_{X_k}} = \frac{(-1)^{d(k+1)-1}}       {d^3k^2}\binom{k^2 d}{d}
    \end{equation}
    where $\mathbf{0}_{X_k}\in H_2(X_k)$ is the class of the zero section.
    Moreover, the associated \emph{BPS} numbers $n_{0,d\mathbf{0}_{X_k}}(X_k)$ satisfy
    \[
    n_{0,d\mathbf{0}_{X_k}}(X_k) = \frac{1}{d(k+1)}D(d, d, k+1)
    \]
    where $D(d,d,m)$ is the diagonal Donaldson--Thomas invariant of degree $d$ of the Kronecker quiver $K_m$.
\end{thm}
%\daniel{Should we also mention \cite{vangarrel2025} in the abstract?}
%\giosue{I think it is not necessary. Moreover, you should cite them by name and not using the latex command, as the abstract must be readable without the article. I changed the abstract accordingly.}

\begin{thm}[Theorem~\ref{thm:Schubert}]
    There exist smooth Schubert varieties $X_w$ and classes $x,y\in H^*(X_w)$ such that the quantum product $x\ast y$ has infinitely many $q$-terms.
\end{thm}

We proved Theorem~\ref{thm:intro_GW_DT_Quiver} directly using the localization formula.
Independently, a completely different proof of this result appeared in \cite[Theorem~3.2]{vangarrel2025}. Their approach uses degenerations of log Gromov--Witten invariants and the relation between refined DT invariants of quivers and log BPS invariants of log Calabi--Yau surfaces given by \cite[Theorem~8.13]{MR4157555}.
%\giosue{See also \cite{MR2304432_HEP}, where a related problem is discussed from a physical point of view.}
Yet another approach would be to use the results of \cite{Bryan_Pandharipande_2008} (see Remark~\ref{rem:BP08}).
An early occurrence of this formula in the physics literature is  \cite[Equation~(4.53)]{MR2304432_HEP}.

The contents of the paper are organized as follows. In Section~\ref{sec:GKM} we review the necessary basic notions of GKM theory and describe the relations between curve classes corresponding to edges of the GKM graph (Proposition~\ref{prop:curve_classes}). 
In Section~\ref{sec:GW_and_QH} we recall the definitions of equivariant Gromov--Witten invariants and quantum cohomology, together with the localization formula of \cite{LS17} and \cite{Hirschi_2023}. We then prove in Theorem~\ref{thm:con_indep} that this formula is independent of the choice of a compatible connection on the GKM graph, and provide an equivalent formulation without reference to such a choice. Section~\ref{sec:tool} introduces the package {\pkg} and explains its main functionalities. In Section~\ref{sec:applications_and_examples} we present a series of applications, ranging from concrete calculations to structural results:
\begin{itemize}
\item In Section~\ref{sec:CY_local_models} we compute equivariant Gromov--Witten invariants (with connected domain curves) of Calabi--Yau rank two vector bundles on $\mathbb{P}^1$, classify the cases where these invariants are rational (Lemma~\ref{lem:polynomiality_condition}), and establish explicit formulas (Theorem~\ref{thm:gw_equivariant_CY}). In the equivariantly Calabi--Yau case we recover in a new way a recent relation with diagonal Donaldson--Thomas invariants of the Kronecker quiver (Section~\ref{sec:BPS}).
\item In Section~\ref{sec:twisted_flag} we give explicit examples of equivariant quantum products for twisted flag manifolds (Theorems~\ref{thm:twisted_Qprod} and~\ref{thm:twisted_numbers}). These are Hamiltonian GKM spaces that do not admit compatible Kähler structures and hence fall strictly within the symplectic, non-algebraic category.
\item In Section~\ref{sec:realizability} we show how equivariant Gromov--Witten invariants can be used to prove non-realizability results for certain GKM graphs, ruling out their occurrence as Hamiltonian or algebraic GKM spaces.
\item In Section~\ref{sec:enumerating_curves} we illustrate how the GKM approach extends to classical enumerative questions, such as the enumeration of curves contained in hyperplanes. 
In particular, we give a complete list of all rational cubics in $\mathbb{P}^4$ contained in a plane and meeting an appropriate number of linear subspaces.
\item In Section~\ref{sec:Schubert} we apply the previous results to show that there are smooth Schubert varieties on which some (equivariant and non-equivariant) quantum product has infinitely many $q$-terms (Theorem~\ref{thm:Schubert}).
We also include an interesting example in type $G_2$ where a certain \emph{equivariant} quantum product has infinitely many terms (Proposition~\ref{prop:Schubert:G2_GB}), but we do not know if there is a non-equivariant quantum product with infinitely many terms for this example.
\end{itemize}

Our results show the strength of GKM theoretic methods in Gromov--Witten theory and provide computational infrastructure to explore further examples.
% Our results \st{clarify foundational aspects} improves of GKM-theoretic curve counting, establish new connections with enumerative invariants, and provide computational infrastructure to explore further examples. 
We expect this combination of theoretical and algorithmic developments to serve as a useful resource for future investigations in symplectic geometry, algebraic geometry, and mathematical physics.

While the examples considered in this paper are mostly restricted to genus zero, we intend to deal with positive genus examples in a follow-up paper.

\begin{ack}
    We sincerely thank Pierrick Bousseau, Jim Bryan, Theresia Eisenk{\"o}lbl, Oliver Goertsches, Tam{\'a}s Hausel, Amanda Hirschi, Chiu-Chu Melissa Liu, Leonardo Mihalcea, Anantadulal Paul, Markus Reineke, Silvia Sabatini, Csaba Schneider, Yannik Schuler, and Bal{\'a}zs Szendr{\H o}i for fruitful discussions and helpful comments.
    The first author did this work during his PhD studies at the Institute of Science and Technology Austria (ISTA), supported by the Austrian Science Fund (FWF) grant no. 10.55776/P35847.
    The second author is a member of GNSAGA (INdAM). Moreover, he is supported by FCT - Funda\c{c}\~{a}o para a Ci\^{e}ncia e a Tecnologia, under the project: UID/04561/2025.
\end{ack}

\section{GKM theory}
\label{sec:GKM}

In this section, we recall the necessary background on GKM theory %(see \cite{GZ01}, \cite{GKZ22} 
(see \cite{GZ01,GKZ22} for more details).
GKM spaces are named after the influential paper by Goresky, Kottwitz, and MacPherson \cite{GKM98}. They provide a large class of examples for which a lot of geometry can be reduced to combinatorial properties of the attached \emph{GKM graph}.

\subsection{GKM spaces}

Throughout this paper, we will deal with two cases simultaneously: the smooth quasi-projective algebraic case and the symplectic Hamiltonian case.
This distinction reflects the fact that Gromov--Witten invariants are defined differently in these two settings. Nevertheless, our results can often be stated in some form that holds in both contexts. The central object of our paper will be a topological space $X$ together with a group action of a group $T$. We call such a structure a $T$-space.
% We will write $\mathfrak{t}$ for the Lie algebra of $T$ and $\mathfrak{t}^*$ for its dual \giosue{but this makes sense if T is a Lie group, you should assume that before. I would move that only in the context of moment map.}. \daniel{Okay, and also in the section of the axial function since it takes values in $\mathfrak{t}^*$}

By variety we mean an integral separable complex scheme of finite type.

By $(\mathbb{C}^\times)^r$ we mean $(\mathbb{C}\setminus\{0\})^{\times r}$.
% \begin{defn}[GKM space]
\begin{defn}%[{\cite[Definition~?]{}}]
    \label{def:GKM_alg}
    A $T$-space $X$ is \emph{algebraic GKM} if
    \begin{enumerate}
        \item $X$ is a smooth quasi-projective variety, $T=(\mathbb{C}^\times)^r$ is a torus and the action is algebraic,\label{item:point_1}
        \item there are only finitely many fixed points, i.e. $0<|X^T|<\infty$, and
        \item for each $p\in X^T$, the $T$-weights of $T_pX$ are pairwise linearly independent.
    \end{enumerate}
\end{defn}

Before defining the notion of \emph{Hamiltonian GKM} spaces, we need to recall the notion of Hamiltonian group action in the context of compact real Lie groups acting on smooth manifolds.

\begin{defn}[{\cite[\S1.1]{Guillemin_Guillemin_Ohsawa_Ginzburg_Karshon_2014}}]
    Let $T$ be a compact real torus acting smoothly on a symplectic manifold $(X,\omega)$. Let $\mathfrak{t}$ be the Lie algebra of $T$ and $\mathfrak{t}^*$ the dual Lie algebra.
    The action is called \emph{Hamiltonian} if it admits a \emph{moment map}, that is, a smooth map $\mu\colon X\rightarrow \mathfrak{t}^*$ such that for any $\xi\in \mathfrak{t}$ we have $\omega(X_\xi,\cdot)=\iota_{X_\xi}\omega=d\langle\mu, \xi\rangle$, where $X_\xi$ is the vector field on $X$ induced by $\xi$.
\end{defn}
A Hamiltonian action automatically preserves the symplectic form $\omega$.
Further important properties of moment maps are recalled in Section~\ref{sec:moment_map}.

We call a $T$-space $X$ \emph{Hamiltonian GKM} 
% \daniel{Maybe "Hamiltonian GKM" would be better?}\giosue{It is ok Hamiltonian. The definition of GKM space varies so much in the literature, that I cannot see any harm calling it Hamiltonian. We just need that our spaces are smooth and non necessary compact.} \daniel{okay! I'm changing everything to Hamiltonian then.} 
if $X$ satisfies Definition~\ref{def:GKM_alg}, replacing condition \eqref{item:point_1} with the following one:
\begin{itemize}
    \item[($\star$)] $X$ is a connected manifold with a symplectic form $\omega$ and a compatible almost complex structure $J$. The action of the torus $T=(S^1)^{\times r}$ is Hamiltonian and preserves $J$.\label{item:point_12}
\end{itemize}
An important consequence of this definition is the following: if $T'\subset T$ is a codimension $1$ subtorus, the connected components of $X^{T'}$ are either points or $2$-spheres \cite[\S1.5]{Guillemin_Zara_2002}. Often, this propriety is part of the definition of GKM manifold, see \cite{GSZ12}.
%\giosue{By Hamiltonian we mean that ...}
%there is a moment map $\mu\colon X\rightarrow\mathfrak{t}^*$ (see Section~\ref{sec:moment_map}).
%\daniel{Move before used in definition of Hamiltonian GKM. Define $\mathfrak{t}$}

\begin{rem}\label{rem:analytification}
    If $X$ is an algebraic GKM space, its analytification 
    $X^\mathrm{an}$ 
    is naturally a Hamiltonian GKM space. It is enough to restrict the action of $T$ to a maximal compact torus $(S^1)^{\times r}\subset T$. We may prove that this action is Hamiltonian using \cite[Examples~8.1(ii)]{MR1304906}. Indeed, if $X$ is an open subset of a projective variety $\bar{X}\subseteq\mathbb{P}^n$, equivariantly embedded via \cite[Theorem~5.1.25]{MR2838836_Chriss_Ginzburg}, the restriction of the Fubini--Study form is a K{\"a}hler form on $X$.  Taking an appropriate choice of coordinates, there exists a linearization $\rho\colon T\rightarrow U(n+1)$ to the unitary group. The moment map (see \cite[Lemma~2.5]{MR766741}) is defined as
    $$\mu (x) a = \frac{\bar {\mathbf{x}}^t \rho_*(a) \mathbf{x}}{2\pi i\, ||\mathbf{x}||^2},$$
    where $a\in \mathfrak{t}$ and $\mathbf{x}$ is a representative vector for $x\in X\subseteq\mathbb{P}^n$.
    
    So the Hamiltonian case includes the algebraic one. We still distinguish these cases because the fundamental definitions of Gromov--Witten invariants are different in each case.
\end{rem}

\begin{defn}\label{def:GKM_space}
    By \emph{GKM space} we mean an algebraic or Hamiltonian GKM space unless otherwise stated.
    By a \emph{$T$-stable $\mathbb{P}^1$} in a GKM space we mean either a $T$-invariant 2-sphere in a Hamiltonian GKM space or a $T$-invariant (complex) $\mathbb{P}^1$ in an algebraic GKM space.
    This is compatible with the analytification construction of Remark~\ref{rem:analytification}.
\end{defn}

\begin{rem}
    A GKM space has only finitely many $T$-stable $\mathbb{P}^1$, see \cite[Proposition~1.1]{MR2064318_Guillemin_Holm_04} and \cite[Proposition~2.3, Remark~2.4]{GKZ22}.%\giosue{ok}
\end{rem}

\begin{rem}
    The terms \emph{GKM space} or \emph{GKM manifold} are often defined in more general settings, where a broader class of spaces is allowed and 
    %\giosue{\st{more} different [if you consider more varieties, you are making less assumptions]}
    different
    assumptions, such as equivariant formality, are made (cf. \cite{GKM98,GKZ22}).
    In our setting, equivariant formality holds by \cite[Proposition~5.8]{MR766741} for Hamiltonian GKM spaces and by \cite[(5.2.2)~Corollary]{Weber_2005} for algebraic GKM spaces.%\giosue{ok}
\end{rem}

\begin{example}\label{example:def_Fl}
    Examples of algebraic GKM spaces include smooth projective toric varieties, and homogeneous spaces $G/P$ where $G$ is a complex reductive algebraic group and $P$ is a parabolic subgroup \cite{Guillemin_Holm_Zara_2006}.
    In type $A_n$, we will write $\mathrm{Fl}(s_1,\dots,s_k)$ for the partial flag variety parametrizing chains
    \[
        0=V_0\subsetneq V_1\subsetneq\dots\subsetneq V_k=\mathbb{C}^{s_1+\dots +s_k}
    \]
    of linear subspaces such that $\dim_\mathbb{C}V_i/V_{i-1}=s_i$ for $i=1,\ldots, k$.
\end{example}

\begin{example}\label{example:Schubert_are_GKM}
    Let $G$ be a connected reductive complex algebraic group, $B$ a Borel subgroup and $T\subseteq B$ a maximal torus. Let $w\in N_G(T)/T$ and $X_w$ be the Schubert variety for $w$. Whenever it is smooth, $X_w$ is a GKM subvariety of $G/B$ because it inherits $T$-stability from the Bruhat decomposition.
    When $X_w$ is singular, there exists a desingularization $Z_{w}\rightarrow X_w$ devised by Bott and Samelson \cite{MR105694} where $Z_w$ is a tower of projective bundles.
    
    If $\underline{w}=(w_1,\dots,w_l)$ is a reduced word for $w$ in the Weyl group, where each $w_i$ corresponds to a simple root $\alpha_i$, the corresponding Bott--Samelson variety is
    \[
    Z_{\underline{w}} := P_1\times_B P_2 \times_B\cdots\times_B P_l/B
    \]
    where $P_1,\dots,P_l$ are the minimal parabolic subgroups corresponding to $\alpha_1,\dots,\alpha_l$ and $B$ acts by right multiplication on the rightmost entry.
    Certain Bott--Samelson varieties are GKM \cite{Withrow_2018}.
\end{example}

% \begin{example}
%     A nondecreasing function $h\colon \{1,2,\ldots, n\}\rightarrow \{1,2,\ldots, n\}$ such that $h(i)\ge i$ for $1\le i\le n$ is called Hessenberg function. 
%     Given a linear operator $X$ on $\mathbb{C}^n$, the variety $\mathcal{H}(X,h)$ of all flags:
%     $$ 0 \subsetneq V_1\subsetneq V_2\subsetneq \cdots\subsetneq V_n=\mathbb{C}^n, $$
%     such that $X(V_i)\subseteq V_{h(i)}$ for all $i$, is called Hessenberg variety \cite{MR1043857}. 
%     Certain Hessenberg varieties are GKM, see \cite{Goldin_Tymoczko_2023}. 
% \end{example}
\begin{example}
    Let $G,B$ and $P$ as in the previous examples, with Lie algebras $\mathfrak{g},\mathfrak{b}$ and $\mathfrak{p}$, respectively. Let also $B\subseteq P$. Let us consider $\mathfrak{g}$ as a $P$-module with respect to the adjoint action. Let $H$ be a $P$-submodule of $\mathfrak{g}$ containing $\mathfrak{p}$, and let $x\in\mathfrak{g}$. The $\mathfrak{p}$-Hessenberg variety of $x$ and $H$ is
    $$\mathrm{Hess}_{\Theta}(x,H):=\left\{gP\in G/P : \mathrm{Ad}(g^{-1})(x)\in H \right\},$$
    where $\Theta$ denotes a subset of the simple roots of $G$ corresponding to $P$. When $\mathfrak{p}=\mathfrak{b}$, those varieties are known simply as Hessenberg varieties \cite{MR1043857}. Certain $\mathfrak{p}$-Hessenberg varieties are GKM, see \cite{Goldin_Tymoczko_2023} and \cite[Section~4]{partialhessenberg}. 
    
    An alternative, less abstract formulation in the $A_n$ case is the following. Let us consider a non-decreasing function $h\colon \{1,2,\ldots, n\}\rightarrow \{1,2,\ldots, n\}$ such that $h(i)\ge i$ for $1\le i\le n$.
    Given a linear operator $X$ on $\mathbb{C}^n$, the Hessenberg variety of $X$ and $h$ is the variety of all flags:
    $$ 0 \subsetneq V_1\subsetneq V_2\subsetneq \cdots\subsetneq V_n=\mathbb{C}^n, $$
    such that $X(V_i)\subseteq V_{h(i)}$ for all $i$.
\end{example}

Many additional intriguing examples arise in the symplectic setting. For instance, a compact symplectic manifold with a Hamiltonian GKM torus action that does not admit a compatible K{\"a}hler structure (see Section~\ref{sec:twisted_flag}). In particular, this manifold is not the analytification of any algebraic GKM variety.

\subsection{The moment map}
\label{sec:moment_map}

We briefly recollect some nice properties of the moment map, which is always present in our setting. When $X$ is an algebraic GKM space, we consider the moment map of its analytification, as explained in Remark~\ref{rem:analytification}.
% \giosue{why?? this is not algebraic}
% \daniel{Because we want the moment map to take values in $\mathbb{R}^{\mathrm{rank}(T)}$ and use the properties below to deduce properties of the GKM graph (like that it has no loops or double edges)}

For Hamiltonian actions by $T=(S^1)^{\times r}$ on a compact symplectic manifold $X$, the moment map is unique up to an additive constant and automatically $T$-invariant \cite[Proposition~2.9]{Guillemin_Guillemin_Ohsawa_Ginzburg_Karshon_2014}.
By the Atiyah--Giullemin--Sternberg convexity theorem, $\mu(X)\subset\mathfrak{t}^*$ is the convex hull of $\mu(X^T)$, the images of the fixed points \cite[\S2.2]{Guillemin_Guillemin_Ohsawa_Ginzburg_Karshon_2014}.
Moreover, the $1$-simplices of the boundary of $\mu(X)$ are images of the $T$-invariant $2$-spheres.
Conversely, each $T$-invariant $2$-sphere gets mapped by $\mu$ to a line segment whose slope is the weight of the $T$-action on the $2$-sphere
(cf. discussion of \emph{local normal form} in \cite[\S2.2]{Guillemin_Guillemin_Ohsawa_Ginzburg_Karshon_2014} and \cite[after Remark~2.11]{GKZ20}).

\subsection{GKM graphs}

The combinatorial device which encodes a lot of a GKM space's geometry and topology is its \emph{GKM graph}.

\begin{defn}[Underlying graph]\label{defn:under_g}
Let $X$ be an algebraic or Hamiltonian GKM space. Assign to it the undirected graph $G$ constructed as follows.
% \st{Note that we allow \emph{flags}, i.e. semi-infinite edges which only have a single vertex.}
\begin{itemize}
    \item The vertices of $G$ correspond to $T$-stable points,
    \item the flags correspond to pairs $(L, v)$ where $v$ is a $T$-stable point, and $L$ is a $T$-invariant irreducible linear summand of $T_{X,v}$,
    \item the edges correspond to pairs $(L_1,v_1), (L_2,v_2)$ of distinct flags for which there exists a $T$-stable $\mathbb{P}^1$ (or $S^2$) in $X$ that contains $v_1$ and $v_2$ and to which $L_1$ and $L_2$ are tangent.
\end{itemize}
\end{defn}
We denote by $V(G),E(G)$ and $F(G)$, respectively, the set of vertices, edges and flags of $G$. We will usually denote an edge by $e$, and by $e=\{v_1,v_2\}$ we mean that $v_1$ and $v_2$ are the two vertices defining $e$.
We will write $C_e$ for the $T$-stable $\mathbb{P}^1$ in $X$ corresponding to $e$.
By $E(G)^\pm$ we denote the set of all flags defined by an edge. Such flags may be denoted as pairs $(e,v_1),(e,v_2)$. We denote by $\overline{(e,v_1)}$ the other flag defined by $e$. Finally, we denote by $F(G)_v$ (resp., $E(G)_v$) the set of all flags (resp., edges) incident to a vertex $v\in V(G)$.
% \begin{rem}\giosue{recheck}
%     In the case of Hamiltonian GKM space, in order to define flags we take real $2$-dimensional $T$-invariant subspaces of $T_X$. So invariant orbits are $2$-dimensional real subspaces, that are our flags in this context. 
%     Such a definition could be known to expert but we could not find a reference. 
% \end{rem}
\begin{example}
    The varieties $\mathbb{P}^1$ and $\mathbb{A}^1$ with the standard $\mathbb{C}^\times$-actions given by $\lambda\cdot [x:y]=[x:\lambda y]$ and $\lambda\cdot z=\lambda z$,
    have, respectively, two and one fixed points. Thus, the underlying graphs for their GKM graphs are the following:
    \begin{figure}[htb]
        \centering
        \includegraphics[width=0.6\linewidth]{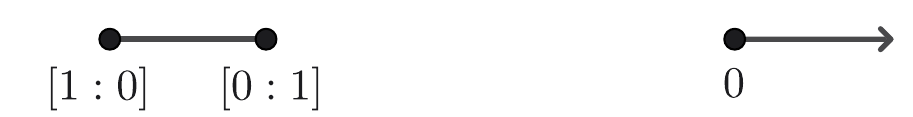}
        \caption{
            The underlying graphs of the GKM graphs of $\mathbb{P}^1$ (left) and $\mathbb{A}^1$ (right).
        }
    \end{figure}
    
    Note that in the first case we have two flags, defining an edge. In the second case we have just one flag, represented by an arrow starting from the vertex representing the fixed point.
\end{example}

% \begin{rem}\giosue{recheck}
%     Flags may only occur if $X$ is not compact.
%     An example of a GKM graph with flags is that of the total space of a GKM vector bundle over a compact GKM space.
%     In that case, the flags at a vertex $p\in X^T$ correspond to $T$-stable linear summands in the fiber over $p$.
% \end{rem}
\begin{rem}
    Let $X$ be an algebraic GKM space. If it is projective, then all flags are in $E(G)^\pm$. %\st{Indeed, 
    % from Lefschetz fixed-point formula \cite[Appendix~C.3.5]{MR0463157}, it follows that 
    % every flag is tangent to a $T$-stable curve [GKZ24, Proposition 2.3] and a $T$-stable projective curve in $X$ has exactly two fixed points.} 
    Indeed, every flag is tangent to a $T$-stable $\mathbb{P}^1$ \cite[Proposition~2.3]{GKZ22} which has exactly two fixed points.
    It defines two flags, and one edge. On the other hand, if $X$ is not projective, then we may have flags not in $E(G)^\pm$, corresponding to $T$-invariant curves whose closure in $X$ is an affine line. An example of such $X$ is the total space of a vector bundle over a GKM space. When $X$ is Hamiltonian, we replace projectivity with compactness.

    % \daniel{The second sentence is not obvious to me. It sounds plausible but would deserve a proof. The reverse implication is clear from the fact that GKM graphs of compact GKM spaces are well-defined without non-edge flags. The rest sounds great!}\giosue{Daniel’s comment is now incorporated. Waiting for Daniel’s opinion.}
    % \daniel{How about this slightly modified version (see blue stuff)?}
\end{rem}

% \daniel{The same arguments as in the Remarks \ref{rem:G_simple} and \ref{rem:G_connected} below were added to the recently published version of \cite{Charton_Kessler_2025} discussing GKM graphs of compact Hamiltonian GKM spaces (it wasn't in the preprint). 
% Should we remove our arguments and only keep the statements, or is the current version good?}\giosue{I read both remark. I think that our Remark 2.13 is very similar to the other one, so it could be compressed to the few lines I added after. The other one seems different. Look how I changed it.}
% \daniel{In fact, I think both remarks are the same Idea. So I reformulated both, just mentioning the idea. In both cases, I would suggest to just keep my last version, if you approve?}

\begin{rem}[$G$ is simple]
% [$G$ is simple, cf. {\cite[Remark~2.7]{Charton_Kessler_2025}}]
% \label{rem:G_simple}
%     Assume that $X$ is compact or projective for this remark.
%     The resulting graph $G$ has no loops at a single vertex because the moment map $\mu$ sends the $T$-stable $\mathbb{P}^1$ to a (non-degenerate) line segment in $\mathfrak{t}^*$ whose endpoints are the images of the fixed points (cf. Section~\ref{sec:moment_map}).
%     There are also no multiple edges between two vertices because the $T$-weight of $T_pS^2$ is given up to scalar multiple by the slope of the line segment $\mu(S^2)$.
%     Thus, having multiple edges between two fixed points would violate the assumed linear independence of the weights.
%     Finally, $G$ is connected as will be shown in Remark \ref{rem:G_connected} after the axial function has been introduced. \giosue{The graph $G$ is simple, see \cite[Remark~2.7]{Charton_Kessler_2025}. Moreover, $G$ is connected as will be shown in Remark \ref{rem:G_connected} after the axial function is introduced.}
    The graph $G$ is simple, i.e. it has no multiple edges or loops, see \cite[Remark~2.7]{Charton_Kessler_2025}. This can be easily shown using the moment map (cf. Section~\ref{sec:moment_map}), which sends each $T$-stable $\mathbb{P}^1$ to a non-degenerate line segment in $\mathfrak{t}^*$ whose endpoints are the images of the fixed points and whose slope is a multiple of the $T$-weight of the $\mathbb{P}^1$.
\end{rem}
\begin{rem}
    If $\dim_\mathbb{R} X = 2n$ then $G$ is $n$-valent, 
    that is at each vertex the number of edges and number of flags sum up to $n$.
\end{rem}

\begin{defn}[Axial function]\label{defn:axial}
    Let $X$ be an algebraic or Hamiltonian GKM space and $G$ the graph in Definition~\ref{defn:under_g}. The \emph{axial function} of $X$,
    \[
    \alpha\colon F(G)\longrightarrow \mathfrak{t}^*,
    \]
    assigns to each flag $(L,v)$ the $T$-weight of $L$.
\end{defn}

% \begin{defn}[Axial function]\giosue{This definition needs to change because $\alpha$ takes values in all flags, not only in the flags coming from full edges }\label{defn:axial}
%     Let $E(G)^\pm$ be the set of edges of $G$ with a choice of an orientation. That is, every $e\in E(G)$ appears in $E(G)^\pm$ twice with its two possible orientations.
%     Given $e\in E(G)$, let $C_e$ be the $T$-stable $S^2$ in $X$ corresponding to $e$.
%     Given $e\in E(G)^\pm$, let $\text{src}(e)$ and $\text{dst}(e)$ be its source and target vertex, respectively.
%     The \emph{axial function}
%     \[
%     \alpha\colon E(G)^{\text{oriented}}\longrightarrow \mathfrak{t}^*
%     \]
%     assigns to each $e\in E(G)^\pm$ the $T$-weight of $T_{\text{src}(e)} C_e$.
%     % It satisfies $\alpha(\overline{e})=-\alpha(e)$, where $\overline{e}$ is the edge $e$ with reversed orientation.
%     % By construction, $\alpha$ takes values in the character lattice $\mathbb{Z}^r\subset \mathfrak{t}^*$ where $r=\text{rank}(T)$.
% \end{defn}

It can easily be seen that if $e=\{v_1,v_2\}$, then $\alpha(e,v_1)=-\alpha(e,v_2)$. 
By construction, $\alpha$ takes values in the character lattice $M\subset \mathfrak{t}^*$ which is isomorphic to $\mathbb{Z}^r$ where $r=\text{rank}(T)$.
\begin{defn}[GKM graph]
    The \emph{GKM graph} of a GKM space $X$ consists of the graph $G$ constructed above together with the axial function $\alpha$.
\end{defn}

\begin{rem}[$G$ is connected]\label{rem:G_connected}
    % Assume again that $X$ is compact.
    % Note that $G$ is connected because $X$ is connected.
    % Indeed, using \cite[\S1.3]{GZ01}, we may calculate the Betti numbers of $X$ by fixing some $\xi\in\mathfrak{t}$ such that $\langle\xi,\alpha(f)\rangle\neq 0$ for all flags $f\in F(G)$.
    % Then the zeroth Betti number $b_0(X)$ is the number of vertices $v$ of $G$ for which all $(e,v)\in E(G)_v$ have $\langle\xi,\alpha(e,v)\rangle>0$.
    % Since $b_0(X)=1$, there exists a unique vertex $p_0\in X^T$ with the above property.
    % Thus, every other vertex $v$ of $G$ has an outgoing edge $(e,v)\in E(G)_v$ with $\langle \xi,\alpha(e)\rangle<0$.
    % Following these edges, we build a path in $G$ from any vertex $v$ to $p_0$, so $G$ is connected.
    % \giosue{See also \cite[Remark~2.9]{Charton_Kessler_2025}.}
    Note that $G$ is connected because $X$ is connected, see \cite[Remark~2.9]{Charton_Kessler_2025}.
    This can be shown either by considering the Chang--Skjelbred Lemma or by using the moment map to connect every vertex to a vertex $p_0$ locally minimizing $\langle\mu(p_0), \xi\rangle$ for some fixed generic $\xi\in\mathfrak{t}$.
    The number of such $p_0$ is the Betti number $b_0(X)=1$ by \cite[\S1.3]{GZ01}.
    Hence, there is a unique such $p_0$ and every vertex is connected via a path to $p_0$.
\end{rem}

The GKM graph constructed from a GKM space is therefore an abstract GKM graph in the following sense.

\begin{defn}[Abstract GKM graph]
    An \emph{abstract GKM graph} of dimension $n$ and torus rank $r$ is a connected unoriented $n$-valent simple graph with flags $G$ with an axial function $\alpha\colon F(G)\rightarrow \mathbb{Z}^r$ satisfying
    \begin{enumerate}
        \item $\alpha(e,v_1)=-\alpha(e,v_2)$ for all $e=\{v_1,v_2\}\in E(G)$ and
        \item for all $(L,v),(L',v)\in F(G)$, the weights $\alpha(L,v)$ and $\alpha(L',v)$ are linearly independent.
    \end{enumerate}
\end{defn}

\begin{example}[Illustration of GKM graphs]
\label{ex:GKM_illustration}
    We will use two ways of illustrating GKM graphs, as explained using the example in Figure~\ref{fig:GKM_drawings}.

    \begin{enumerate}
        \item On the left hand side of Figure~\ref{fig:GKM_drawings}, some oriented edges are labelled with their value of $\alpha$.
        To determine $\alpha$ for all flags, we impose the rule that, unless stated otherwise, parallel flags with the same direction have the same value of $\alpha$, and that $\alpha(e,v_1)=-\alpha(e,v_2)$ for any edge $e=\{v_1,v_2\}$.

        \item On the right hand side of Figure~\ref{fig:GKM_drawings}, we use an illustration method that works when $\mathrm{rank}(T)=2$.
        Unless stated otherwise, for any flag $f\in F(G)$, the value $\alpha(f)$ is given by the primitive integral vector pointing in the direction of $f$.
    \end{enumerate}

    \begin{figure}[htb]
        \centering
        \includegraphics[width=0.4\linewidth]{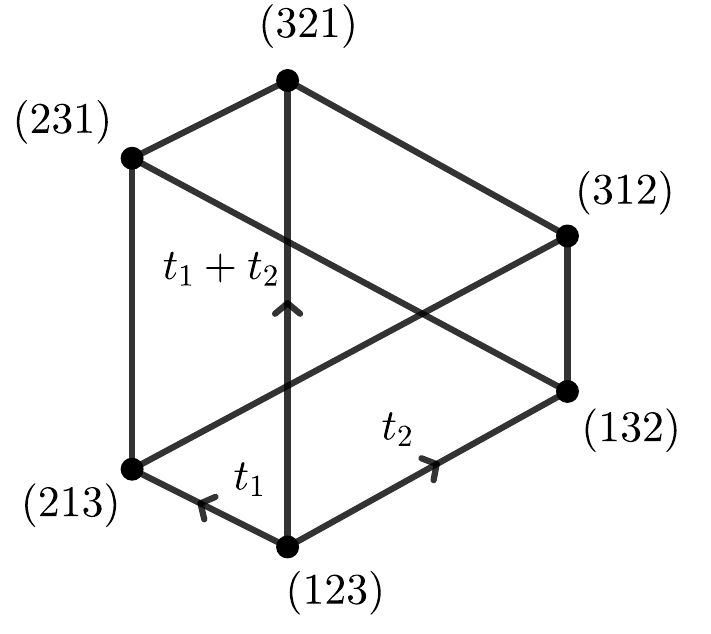}
        \hspace{1cm}
        \includegraphics[width=0.4\linewidth]{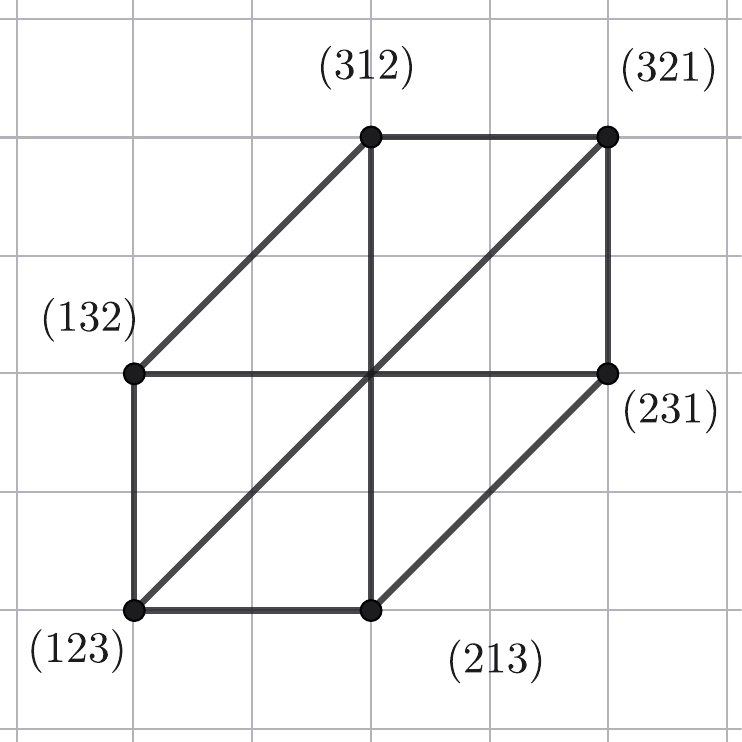}
        \caption{
            Two ways of depicting the GKM graph of the full flag variety for $\mathbb{C}^3$ acted on by  $T=(\mathbb{C}^\times)^2$.
            It has six vertices, labeled by the elements of $S_3$.
            See also \cite[Fig.~3~(left)]{Guillemin_Holm_Zara_2006}.
        }
        \label{fig:GKM_drawings}
    \end{figure}
\end{example}

\subsection{Equivariant cohomology}

We briefly recall how to express equivariant cohomology of GKM spaces and how to perform equivariant integration using the GKM graph.
We refer to \cite{Hsiang_1975, Ati82} and \cite{MR4655919} for background on equivariant cohomology and equivariant integration.

We write $H_T^*:=H_T^*(\{\mathrm{pt}\})\cong \mathbb{Z}[t_1,\dots,t_r]\cong\mathbb{Z}[\mathfrak{t}]$, where $r=\mathrm{rank}(T)$ and $t_1,\dots,t_r$ are the \emph{equivariant parameters}.
Using the above isomorphism, we may think of $t_1,\dots,t_r$ as elements of $\mathfrak{t}^*$.
We have the \emph{restriction map} $H_T^*(X)\rightarrow H^*(X)$ and the \emph{integration maps} $\int_X^T\colon H_T^*(X)\rightarrow H_T^{*}$ and $\int\colon H^*(X)\rightarrow H^*(\{pt\})$ of degree $-\dim_{\mathbb{R}}X$ that are compatible with restriction.
The map $H_T^*\rightarrow H^*(\{pt\})$ is given by setting the equivariant parameters to zero.
For any $x\in X^T$ we have a $T$-equivariant map $\{x\}\rightarrow X$ and hence a \emph{localization map} $H_T^*(X)\rightarrow H_T^*$.
For $\gamma\in H_T^*(X)$, we write $\gamma|_x\in H_T^*$ for the image of $\gamma$ under this localization map.

The following \emph{GKM theorem} originates from \cite[Theorem~1.2.2]{GKM98}, see also \cite[Theorem~1.7.3]{GZ01} and references therein.

\begin{thm}
\label{thm:GKM}
    Let $X$ be a GKM space.
    Then the localization map
    \[
    H_T^*(X;\mathbb{Q})\longrightarrow\bigoplus_{v\in V(G)} \mathbb{Q}[\mathfrak{t}]
    \]
    is injective.
    Its image is given by all tuples $(f_v)_{v\in V(G)}$ such that for all edges $e=\{p,q\}\in E(G)$, we have $\alpha(e,p)\mid f_{p}-f_{q}$.
\end{thm}

The following localization formula for equivariant integration is due to Atiyah--Bott \cite[Equation~(3.8)]{Atiyah_Bott_1984} and Berline--Vergne \cite{Berline_Vergne_1983}.
We state it for GKM spaces although it is of course much more general.
\begin{thm}\label{thm:ABBV}
    Let $X$ be either a projective algebraic or a compact Hamiltonian GKM space, and let $G$ be its GKM graph.
    Let $f\in H_T^*(X)$ correspond to $(f_v)_{v\in V(G)}$ under the identification of Theorem~\ref{thm:GKM}.
    Then
    \[
    \int_X^T f = \sum_{v\in V(G)}\frac{f_v}{\prod_{\epsilon\in F(G)_v}\alpha(\epsilon)}.
    \]
    % where $E(G)_v:=\{e\in E(G)^\pm:\mathrm{src}(e)=v\}$.
    % \daniel{Check if $E(G)_v$ has been removed already.}\giosue{It has ben introduced $F(G)_v$}\daniel{Ok!}

    % \daniel{About your comment in the formula: You are right, it should be $F(G)_v$, although I suspect that the classical references above only consider compact spaces. }
    % \giosue{so we may use $F(G)_v$ and add the hypothesis that $X$ is projective. This theorem is not really used anywhere. ok?}
    % \daniel{Yes! I would say compact instead of projective, because it includes compact manifolds that are not projective. In particular, we have Hamiltonian GKM spaces in our paper which are not projective.
    % Also the theorem is used in the proof of the Proposition~\ref{prop:curve_classes}.}
\end{thm}
\begin{rem}
    The right hand side of the localization formula above lives purely in $H_T^*\cong \mathbb{Z}[t_1,\dots,t_r]$.
    It therefore gives rise to rational functions in $t_1,\dots,t_r$ that simplify non-trivially to a polynomial.
    Geometrically, this is clear because $\int_X^T$ takes values in $H_T^*$.
    However, the same polynomiality result still holds if the right hand side is calculated for an abstract GKM graph $G$ and $(f_v)$ as in Theorem \ref{thm:GKM}, without assuming the existence of a realizing space $X$, see \cite[Theorem~2.2]{Guillemin_Zara_2002}.
\end{rem}

\begin{rem}[First Chern class]\label{rem:c1}
    Let $X$ be a GKM space with GKM graph $G$.
    The equivariant first Chern class $c_1^T(T_X)\in H_T^2(X)$ is determined by $G$ via
    \[
    c_1^T(T_X)|_v = \sum_{f\in F(G)_v} \alpha(f)\hspace{7mm}\forall v\in V(G).
    \]
    Hence, using Theorem~\ref{thm:GKM}, the expression $c_1^T(T_X)$ is well-defined in terms of the abstract GKM graph $G$ even in the absence of a realizing GKM space $X$. The same is true for $c_1(T_X)$ by applying the forgetful map.
\end{rem}

\subsection{Connections}\label{sec:connections}

Given a GKM space $X$ and an edge $e=\{p,q\}$, the Birkhoff--Grothendieck theorem \cite[Th\'eor\`eme~2.1]{MR87176} gives a splitting
\[
T_X|_{C_e}\cong \mathcal{O}_{\mathbb{P}^1}(a_1)\oplus\dots\oplus\mathcal{O}_{\mathbb{P}^1}(a_n).
\]
Restricting a summand to the two fixed points $p,q\in C_e$ determines a (complex) one-dimensional $T$-stable subspace of $T_{C_e,p}$ and $T_{C_e,q}$, respectively.
By pairwise linear independence of the weights, those subspaces correspond precisely to flags of $G$ at $p$ and $q$, respectively.
Thus, one obtains a bijective map
\[
\nabla_{(e,p)}\colon F(G)_p\longrightarrow F(G)_q.
\]
The collection $\{\nabla_{(e,p)} : {(e,p)}\in E(G)^\pm\}$ forms a compatible connection in the following sense.

\begin{defn}%[Compatible connection]
\label{def:connection}
    A \emph{compatible connection} on an abstract GKM graph $G$ with axial function $\alpha$ is a collection of bijective maps $\nabla_{(e,p)}\colon F(G)_p\rightarrow F(G)_q$ for all $e=\{p,q\}\in E(G)$ satisfying
    \begin{enumerate}
        \item $\nabla_{(e,q)}=\nabla_{(e,p)}^{-1}$ for all $e=\{p,q\}\in E(G)$, and
        \item for every $(e',p)\in E(G)_p$ there exists $a\in \mathbb{Z}$ such that
        \[
        \alpha(\nabla_{(e,p)} (e',p)) = \alpha(e',p) - a \alpha(e,p).
        \]
        \label{def:connection:as}
    \end{enumerate}
\end{defn}
\begin{defn}[Induced connection]\label{def:induced_connection}
    We will refer to the compatible connection constructed above using the Birkhoff--Grothendieck theorem as the \emph{compatible connection on $G$ induced by $X$}.
\end{defn}

\begin{defn}[$k$-independence]
    An abstract GKM graph $(G,\alpha)$ is called \emph{$k$-independent} if for each $p\in V(G)$ and any distinct $\epsilon_1,\dots,\epsilon_k\in F(G)_p$, the weights $\alpha(\epsilon_1),\dots,\alpha(\epsilon_k)$ are linearly independent.
\end{defn}
\begin{rem}\label{rem:computing_conn}
    It follows that every GKM graph is $2$-independent by definition.
    If a GKM graph is $3$-independent, one verifies easily that a compatible connection is unique if it exists.
    On the other hand, when 3-independence fails, there can be many compatible connections. Note that the Birkhoff--Grothendieck theorem only implies that $X$ induces a compatible connection on its GKM graph. To the best of the authors' knowledge, there is no general way to determine the induced connection algorithmically.
\end{rem}

\begin{example}[Non-uniqueness of compatible connections]
    Consider the GKM graph $G$ of the full flag manifold of $\mathbb{C}^3$ (see Figure~\ref{fig:GKM_drawings}).
    It is $2$-independent but not $3$-independent because the acting torus has rank $2$.
    As abstract GKM graph, $G$ admits multiple compatible connections.
    For example, let $e$ be the edge $\{(123), (213)\}$.
    Then there are two bijections $\nabla_{(e,(123))}\colon F(G)_{(123)}\rightarrow F(G)_{(213)}$ and one easily checks that both satisfy Definition~\ref{def:connection}~(\ref{def:connection:as}).
    As Definition~\ref{def:connection} does not impose any compatibility conditions between $\nabla_{(e_1,p_1)}$ and $\nabla_{(e_2,p_1)}$ for $e_1\neq e_2$, we can set $\nabla_{\overline{(e,(123))}}:=\nabla_{(e,(123))}^{-1}$ and define $\nabla_{(e',p)}$ to coincide with the connection on $G$  induced by the flag manifold realization (Definition~\ref{def:induced_connection}) for any $e'\neq e$.
    Thus, we obtain two compatible connections on $G$ that differ precisely at $e$.
\end{example}

\begin{rem}
    Some authors include in the definition of an abstract GKM graph that it should admit some compatible connection.
    Whenever this is assumed, it is important to stress for each combinatorial construction whether it depends on the particular choice of connection.
    In particular, the localization formula for Gromov--Witten invariants of GKM spaces is originally formulated using a specific choice of compatible connection, namely the one induced by the GKM space (see Section~\ref{sec:localization_formula}).
\end{rem}

The next definition follows \cite[Definition~4.1]{Charton_Kessler_2025}.
It is an easy exercise to show that it is well-defined.

\begin{defn}\label{def:C1}
    Let $G$ be an abstract GKM graph that admits a compatible connection.
    The \emph{Chern number} map $\mathcal{C}_1\colon E(G)\rightarrow\mathbb{Z}$ is given by
    \[
        e=\{v_1,v_2\} \longmapsto  \frac{1}{\alpha(e,v_1)}\left(\sum_{f_1\in F(G)_{v_1}}\alpha(f_1) - \sum_{f_2\in F(G)_{v_2}}\alpha(f_2)\right).
    \]
\end{defn}
\begin{rem}
    $\mathcal{C}_1$ is well-defined by Definition~\ref{def:connection}~\eqref{def:connection:as}.
    If $G$ is realized by a GKM space $X$ then $\mathcal{C}_1(e)=\int_{C_e}c_1(T_X)$ by Theorem~\ref{thm:ABBV} and Remark~\ref{rem:c1}.
    Moreover, given any compatible connection for $G$, we have
    $\mathcal{C}_1(e)=a_1+\cdots+a_n$ where $a_1,\dots,a_n$ are the integers associated to $e$ via Definition \ref{def:connection} \eqref{def:connection:as} and $n$ is the valency of $G$.
\end{rem}

\subsection{Curve classes}

Another problem in calculating Gromov--Witten invariants purely in terms of the GKM graph using the localization formula is that one needs to understand the relations between the classes $[C_e]\in H_2(X)$ for all $e\in E(G)$.

The following lemma is well-known to experts.
For Hamiltonian GKM spaces, a proof using Morse theory may be found in \cite[Theorem~4.5(i)]{Li_2017}.
For algebraic GKM spaces, a similar proof works using the Bia\l{}ynicki--Birula decompostion \cite[Lemma~4.1 \& Theorem~4.4]{Car02}, or one can apply the Hamiltonian proof to the analytification (see Remark~\ref{rem:analytification}).

\begin{lem}
\label{lem:torsion_free}
    Let $X$ be an algebraic or Hamiltonian GKM space.
    Then $H_*(X;\mathbb{Z})$ is torsion-free.
\end{lem}
%\begin{proof}
%We present two proofs: one for algebraic GKM spaces and one for Hamiltonian GKM spaces.
%
%First, let $X$ be an algebraic GKM space.
%The Bia\l{}ynicki--Birula decompostion \cite[Lemma~4.1 \& Theorem~4.4]{Car02} gives
%\[
%H_m(X;\mathbb{Z})\cong \bigoplus_{v\in V(G)} H_{m-2p(v)}(p_v;\mathbb{Z}).
%\]
%where $p(v)=\dim_\mathbb{C} T^+_{X,p_v} \in\mathbb{Z}_{\ge 0}$.
%\daniel{Say what $T^+$ is.}
%% $p(v)=\dim_\mathbb{C} T^+_{p_v} X\in\mathbb{Z}_\ge 0$.
%The lemma immediately follows because each direct summand on the right hand side is just %a copy of $\mathbb{Z}$.
%
%Second, let $X$ be a Hamiltonian GKM space.
%Then following \cite[proof~of~Theorem 1.3.2] {GZ01} and \cite[Lemma~2.2]{Ati82}, one may %choose a Hamiltonian function $f$ that is Morse and whose critical points all have even %index.
%Using gradient-like vector fields we may ensure that all critical points have distinct %values without changing the indices.
%Finally, we may pick a generic metric $g$ making $f$ Morse--Smale.
%Thus, $H_*(X;\mathbb{Z})$ is isomorphic to Morse homology for $X$ computed using $f$.
%In the underlying chain complex, all differentials vanish because all indices are even.
%In particular, there is no torsion.
%\end{proof}

The next Proposition gives us an explicit way of computing a complete set of relations in $H_2(X)$ among $\{[C_e]:e\in E(G)\}$.

\begin{defn}
    Let $G$ be a graph. A \emph{closed loop} is a finite sequence of flags $\{(e_i,p_i)\}_{i=1}^m \subset E(G)^\pm$ such that
    \begin{itemize}
        \item $e_i=\{p_i,q_i\}$ for $1\le i \le m$,
        \item $q_i=p_{i+1}$ for $1\le i \le m-1$,
        \item $q_m=p_1$.
    \end{itemize}
\end{defn}

\begin{prop}
\label{prop:curve_classes}
    Let $X$ be either a projective algebraic or a compact Hamiltonian GKM space, and let $G$ be its GKM graph. 
    Let $\{1_e\}_{e\in E(G)}$ be the standard basis of the $\mathbb{Z}$-module $\mathbb{Z}^{E(G)}$.
    For each closed loop $\gamma=\{\epsilon_i\}_{i=1}^m$ in $G$, write  $\epsilon_i=(e_i,p_i)$ and define the linear map
    \begin{align*}
        R_\gamma\colon \mathfrak{t}&\longrightarrow \mathbb{C}^{E(G)}
        \\
        \xi &\longmapsto \sum_{i=1}^m \langle \alpha(\epsilon_i),\xi\rangle 1_{e_i}.
    \end{align*}
    Then the kernel of the $\mathbb{Z}$-module homomorphism
    \begin{align*}
        Q\colon \mathbb{Z}^{E(G)}&\longrightarrow H_2(X;\mathbb{Z})\\
        1_e&\longmapsto [C_e]
    \end{align*}
    is $K\cap \mathbb{Z}^{E(G)}$, where 
    \[
    K:=\mathrm{span}_\mathbb{C}\left\langle \mathrm{Im}(R_\gamma) : \gamma \text{ is a closed loop in $G$}\right\rangle\le \mathbb{C}^{E(G)}.
    \]
    Moreover, to generate $K$ it suffices to let $\gamma$ range over any set of cycles generating $H_1(|G|;\mathbb{Z})$, where $|G|$ is the geometric realization of $G$.

    % \giosue{A lot of confusion between E(G) and $E(G)^\pm$. Revise}

    % \daniel{Should be fine now. The important point was that $r_e$ really only depends on the unoriented edge $e$, while $\alpha(\epsilon)$ depends on the choice of vertex.
    % So indeed, we always have $\mathbb{C}^{E(G)}$ or $\mathbb{Z}^{E(G)}$ appearing, but never a $\mathbb{C}^{E(G)^\pm}$ or $\mathbb{Z}^{E(G)^\pm}$.}
\end{prop}

%\begin{rem}
%    \daniel{In the algebraic case:} By an argument like in [Withrow, Bott--Samelson varieties, p.8], one can show using the Borel fixed %point theorem on Chow varieties that $\text{Im}(q)\le H_2(X;\mathbb{Z})$ is the subgroup given by classes represented by %algebraic curves.
%\end{rem}

\begin{proof}
By the Universal Coefficients Theorem and torsion-freeness (Lemma~\ref{lem:torsion_free}), an element $c\in H_2(X;\mathbb{Z})$ is zero if and only if $\varphi(c)=0$ for all $\varphi\in H^2(Z;\mathbb{C})$.
By equivariant formality, this is equivalent to $\psi(c)=0$ for all $\psi\in H_T^2(X;\mathbb{C})$.

By applying Theorem~\ref{thm:GKM}, an element $\psi\in H_T^2(X;\mathbb{C})$ is equivalent to a collection $(f_v)_{v\in V(G)}$ of linear polynomials $f_v\in \mathbb{C}[t_1,\dots,t_r]$ such that $\alpha(e,p)$ divides the polynomial $f_{p} - f_{q}$ %$\alpha(e)\mid f_v - f_w$ 
for all $e=\{p,q\}\in E(G)$.
Since both are linear in $t_1,\dots,t_r$, 
there must be constants $(r_e)_{e\in E(G)}$ such that
% there is $(r_e)\in\mathbb{C}^{E(G)}$ such that
\[
f_{p}-f_{q} = r_e \alpha(e,p)
\]
for all edges $e=\{p,q\}$.
Note that $r_e$ does not depend on the orientation of $e$ because $\alpha(e,p)=-\alpha(e,q)$.
Then $(r_e)$ automatically satisfies
%\[
%\sum_{i=1}^m %r_{\epsilon_i}\alpha(\epsilon_i)=0\in\mathfrak{t}^*\hspace{5mm} \text{  %for all closed loops } \{\epsilon_i\}_{i=1}^m \text{ of oriented edges %in } G.
%\]
\[
\sum_{i=1}^m r_{e_i}\alpha(\epsilon_i)=0\in\mathfrak{t}^*\hspace{5mm} \text{ for all closed loops } \{\epsilon_i=(e_i,p_i)\}_{i=1}^m \text{  in } G.
\]
Note that this is equivalent to
\begin{equation}
\label{eqn:loopCondition}
\left\langle (r_e), R_\gamma(\xi)\right\rangle = 0
\hspace{5mm}\text{ for all closed loops }\gamma\text{ in }G \text{ and }\xi\in \mathfrak{t},
\end{equation}
where 
%$(r_\epsilon)\in\mathbb{C}^{E(G)}$ and
$\langle\cdot,\cdot\rangle$ is the standard bilinear form on $\mathbb{C}^{E(G)}$.
Conversely, any $(r_e)\in\mathbb{C}^{E(G)}$ satisfying \eqref{eqn:loopCondition} determines a non-empty family of elements $\psi\in H_T^2(X;\mathbb{C})$.
By Theorem~\ref{thm:ABBV},
\[
\psi([C_e])=\int_{C_e}\psi = \frac{f_{p} - f_{q}}{\alpha(e,p)} = r_e
\]
for all $e=\{p,q\}\in E(G)$.
Therefore, given $(d_e)\in\mathbb{Z}^{E(G)}$, we have
\begin{align*}
    Q((d_e))=0 & & \Longleftrightarrow & & \langle(d_e),(r_e)\rangle = 0 \hspace{5mm}\forall (r_e)\in\mathbb{C}^{E(G)} \text{ satisfying \eqref{eqn:loopCondition}}.
\end{align*}
% \[
% q((a_e))=0\hspace{5mm}\Longleftrightarrow\hspace{5mm}
% \sum_{e\in E(G)} a_e r_e = 0 \hspace{5mm}\forall (r_e)\in\mathbb{Q}^{E(G)} \text{ satisfying (\ref{eqn:loopCondition})}.
% \]
In other words, $\text{Ker}(Q) = \mathbb{Z}^{E(G)}\cap (K^\perp)^\perp=\mathbb{Z}^{E(G)}\cap K$, as required.

The \emph{moreover} part follows because $R_{\gamma_1\cup \gamma_2} = R_{\gamma_1}+R_{\gamma_2}$, where $\gamma_1\cup \gamma_2$ denotes the concatenation of two loops $\gamma_1$ and $\gamma_2$ in $G$.
\end{proof}

\begin{example}
\label{ex:curve_classes}
Figure~\ref{fig:ex:curve_classes} illustrates Proposition~\ref{prop:curve_classes} for the GKM graph $G$ of $\mathbb{P}^2$ acted on by $T=(\mathbb{C}^\times)^3$ via $(\lambda_1,\lambda_2,\lambda_3)\cdot[x_1:x_2:x_3] = (\lambda_1 x_1, \lambda_2 x_2, \lambda_3 x_3)$.
Let $\partial_{\lambda_1},\partial_{\lambda_2},\partial_{\lambda_3}$ be the generators of $\mathfrak{t}$ corresponding to $\lambda_1,\lambda_2,\lambda_3$, respectively, and let $t_1,t_2,t_3\in\mathfrak{t}^*$ be the equivariant parameters corresponding to $\lambda_1,\lambda_2,\lambda_3$.

In this example, $H_1(|G|)$ is generated by the loop $\gamma$ depicted on the left hand side of Figure~\ref{fig:ex:curve_classes} in red.
The top right relation is obtained from $R_\gamma(\partial_{\lambda_1})$, while the bottom right relation is obtained from $R_\gamma(\partial_{\lambda_2})$.
The relation from $R_\gamma(\partial_{\lambda_3})$ is similar but follows from the previous two.
Note that reading off these relations simply corresponds to reading off the coefficients of the axial function along $\gamma$.

Hence, all three edges represent the same curve class in $G$.

    \begin{figure}[htb]
        \centering
        \includegraphics[width=0.55\linewidth]{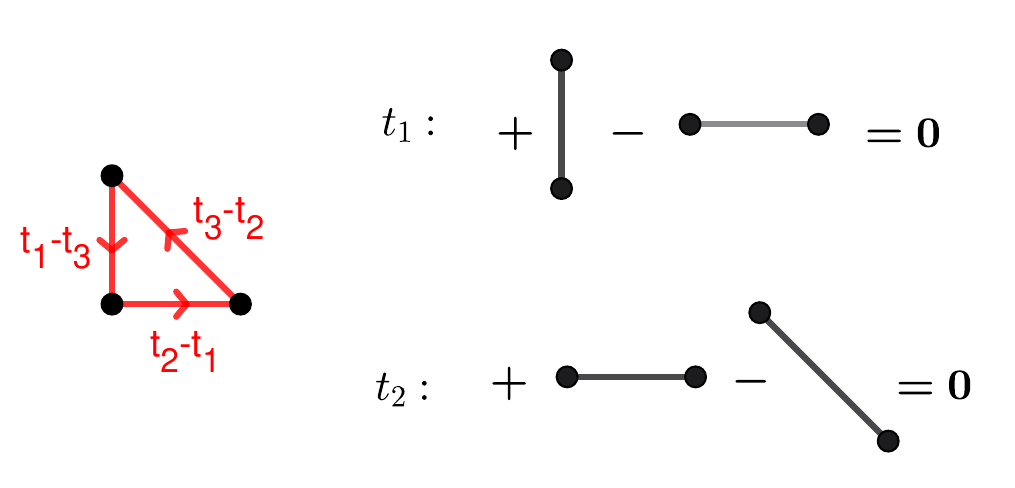}
        \caption{Illustration of Proposition~\ref{prop:curve_classes} for the GKM graph of $\mathbb{P}^2$ with the diagonal action of $(\mathbb{C}^\times)^3$.
        In the equations on the right hand side, we draw an edge to indicate its class in $H_2(\mathbb{P}^2)$.}
        \label{fig:ex:curve_classes}
    \end{figure}
\end{example}

\begin{rem}[Number of relations]
The method derived from Proposition~\ref{prop:curve_classes} produces $\mathrm{rank}(T)\cdot\dim_\mathbb{Q}H_1(|G|;\mathbb{Q})$ relations among the $|E(G)|$ generators of $\mathbb{C}^{|E(G)|}$.
We also have
\[
\dim_\mathbb{Q}H_1(|G|;\mathbb{Q}) = 1 + \frac{\dim_\mathbb{C}(X)\cdot|V(G)|}{2} - |V(G)|
\]
which is obtained using $|E(G)|=\dim_\mathbb{C}(X)\cdot|V(G)|/2$ and calculating the Euler characteristic of $|G|$ once via homology groups and once via chain complexes.
Thus, Proposition~\ref{prop:curve_classes} produces
\[
\mathrm{rank}(T)\cdot\left(1 + |V(G)|\cdot(\dim_\mathbb{C}(X)/2 -1)\right)
\]
relations among the $\dim_\mathbb{C}(X)\cdot|V(G)|/2$ generators of $\mathbb{C}^{E(G)}$.
Comparing these quantities, we see that the number of relations is generally much higher than the number of generators, and hence there are many relations between the relations.
\end{rem}

\begin{rem}[Alternative approach using Kirwan classes]
    To obtain an optimal number of relations, one may use the degree two elements $\{\psi_1,\dots,\psi_{b_2(X)}\}$ of any $H_T^*$-linear homogeneous basis of $H_T^*(X;\mathbb{C})$ and calculate their annihilator in $\mathbb{C}^{E(G)}$.
    A natural class of such bases is given by \emph{Kirwan classes} (named after \cite{MR766741}, cf. \cite[Lemma~2.3]{Godinho_Sabatini_2014}).
    Picking any real $\xi\in \mathfrak{t}$ with $\langle\xi,\alpha(\epsilon)\rangle\neq 0$ for all $\epsilon\in E(G)^\pm$, the \emph{index} $\lambda^\xi(v)$ of $v$ is the number of $\epsilon\in E(G)_v$ with $\langle \xi,\alpha(\epsilon)\rangle < 0$.
    A \emph{Kirwan class at $v\in V(G)$} is an element $\gamma\in H_T^{2\lambda^\xi(v)}(X;\mathbb{C})$ such that:
    \begin{enumerate}
        \item $\gamma|_v = \prod\alpha(e,v)$ where the product is over $\{e\in E(G)_v : \langle\xi,\alpha(e,v)\rangle < 0\}$
        \item $\gamma|_w=0$ whenever $w\in V(G)\setminus\{v\}$ with $\langle\xi,\mu(w)\rangle \le \langle\xi,\mu(v)\rangle$.
    \end{enumerate}
    Any collection $\{\psi_v\}_{v\in V(G)}$ in which $\psi_v$ is a Kirwan class at $v$ for all $v\in V(G)$ is automatically a $H_T^*$-basis of $H_T^*(X;\mathbb{C})$.
    Kirwan classes always exist \cite[Theorem~2.4.2]{GZ01} and can be computed algorithmically \cite[Lemma~2.4]{Godinho_Sabatini_2014}.
    Hence, a complete and minimal set of relations for $\{[C_e]:e\in E(G)\}$ can be obtained by computing a Kirwan class for each $v\in V(G)$ of index one.
\end{rem}
One advantage of our approach in Proposition~\ref{prop:curve_classes} is that the relations can be read off graphically directly from the GKM graph once one has obtained a basis of $H_1(|G|;\mathbb{Z})$, as seen in Example \ref{ex:curve_classes}.

%I added a new remark to Lemma 5.2. ok?
%Ah yes, let me respond there!
% Have you seen my question in Lemma 5.2?
% Oh shit, you are right! I think the argument really only works for $k\ge 2$.
% For $k=1$ we only get $t_3 = -t_1+\mathrm{anything}\cdot t_2$.
% I don't know what the GW invariants are in this case. Let me run some experiments to see if it is an interesting case!
% I guess this also answers my comment in the proof of Lemma 5.2: "YES, I should do this in more detail."

\section{Gromov--Witten invariants and quantum cohomology}
\label{sec:GW_and_QH}

\subsection{Equivariant Gromov--Witten invariants}\label{sec:GW}

Let $X$ be a smooth projective variety and $\beta\in H_2(X)$.
Let $\overline{\mathcal{M}}_{g,n}(X;\beta)$ be the moduli space of stable maps\footnote{
    In particular, $C$ should be connected.
} $f\colon C\rightarrow X$ of genus $g$ with $n$ marked points and $f_*[C]=\beta$ \cite[\S7.1.1]{Cox_Katz_1999}.
The $T$-action on $X$ induces a $T$-action on $\overline{\mathcal{M}}_{g,n}(X;\beta)$.
Let
\[
\ev_i\colon \overline{\mathcal{M}}_{g,n}(X;\beta) \longrightarrow X
\]
be the $T$-equivariant evaluation morphism at the $i$-th marked points.
The \emph{equivariant Gromov--Witten invariants} associated to $(X,g,n,\beta)$ are given by
\begin{align*}
    GW_{g,n}^{X,\beta}\colon H_T^*(X;\mathbb{Q})^{\otimes n} &\longrightarrow  H_T^*(\text{pt};\mathbb{Q})\cong \mathbb{Q}[t_1,\dots,t_r]
    \\
    \alpha_1 \otimes\dots\otimes\alpha_n & \longmapsto \int^T_{[\overline{\mathcal{M}}_{g,n}(X;\beta)]^\text{vir}} \ev_1^*(\alpha_1)\smile\dots\smile\ev_n^*(\alpha_n)
\end{align*}
where $\int^T$ denotes the $T$-equivariant integral (proper pushforward to a point) and $[\overline{\mathcal{M}}_{g,n}(X;\beta)]^\text{vir}$ is the virtual fundamental class \cite[\S7.1.4]{Cox_Katz_1999}.

%\daniel{Should we have a short chapter introducing equivariant cohomology and integration?}
%\giosue{If we can find a good reference, we do not need it. We can just say "we denote by $H_T$ the equivariant cohomology, see [A ref].}

The virtual dimension of $\overline{\mathcal{M}}_{g,n}(X;\beta)$ is 
\[
\mathrm{vdim}(\overline{\mathcal{M}}_{g,n}(X;\beta)) = (1-g)(\dim(X) - 3)
 + \int_\beta c_1(T_X)+n
\]
and $GW_{g,n}^{X,\beta}$ is homogeneous of degree $-\mathrm{vdim}$ (in complex units where $\deg t_i = 1$).
By equivariant formality, setting the equivariant parameters to zero recovers the non-equivariant Gromov--Witten invariants.

In our symplectic setting, we will use the definition of equivariant Gromov--Witten invariants via global Kuranishi charts \cite{Abouzaid_McLean_Smith_2021,MR4807086}, because the same localization formula is available by \cite{Hirschi_2023} (see Section~\ref{sec:localization_formula}).
Note that this definition agrees with the one via pseudocycles when $(X,\omega)$ is semipositive \cite[Theorem~6.3]{Hirschi_2023}.

In the algebraic as well as the symplectic setting, Gromov--Witten invariants satisfy the Kontsevich--Manin axioms, see \cite[\S7.3.1]{Cox_Katz_1999}, \cite[\S3]{Hirschi_2023}, and references therein.
Together with equivariant localization (see Section \ref{sec:localization_formula}), they provide the main means of calculating Gromov--Witten invariants.

When $X$ is a non-compact algebraic GKM space, its $T$-equivariant Gromov--Witten invariants are still well-defined via virtual localization (see \cite[\S3.5.2]{LS17} and Section~\ref{sec:localization_formula} below). In this case, they lie in $\mathrm{Frac}(H_T^*(\mathrm{pt};\mathbb{Q}))\cong\mathbb{Q}(t_1,\dots,t_r)$.

\subsection{Equivariant quantum cohomology}
\label{sec:equivariant_QH}

An important construction from equivariant Gromov--Witten invariants is the \emph{equivariant quantum cohomology} ring.

\begin{defn}
    The \emph{equivariant quantum cohomology} is given as module by
    \[
        QH_T^*(X) := H_T^*(X;\mathbb{Q})\otimes_\mathbb{Q} \Lambda
    \]
    where $\Lambda = \widehat{\mathbb{Q}[H_2^\text{eff}(X)]}$ is the semigroup ring of the curve classes $H_2^\text{eff}(X)\subset H_2(X)$.
\end{defn}
We will write $q^\beta\in \Lambda$ for the element corresponding to $\beta\in H_2^\text{eff}(X)$ so that $q^0 = 1$ and $q^{\beta_1}q^{\beta_2}=q^{\beta_1+\beta_2}$.
Furthermore, we grade $q^\beta$ by $\int_\beta c_1(T_X)$ in complex units.
A detailed discussion of other common choices for the quantum coefficient ring $\Lambda$ may be found in \cite[\S11.1]{McDuff_Salamon_2012}.

\begin{defn}
    The \emph{equivariant (small) quantum product} $\ast_T$ on $QH_T^*(X)$ is defined for $a,b\in H_T^*(X;\mathbb{Q})$ by
    \[
    a\ast_T b := \sum_{\beta\in H_2^\text{eff}(X)} GW_{0,3}^{X,\beta}(a,b,e_i)e^i q^\beta,
    \]
    where $(e_i)_i$ and $(e^i)_i$ are dual $H_T^*$-linear bases of $H_T^*(X)$ with respect to the $H_T^*$-linear pairing $\langle a,b\rangle := \int^T_X a\smile b$.
    Equivalently, $\ast_T$ is defined to satisfy
    \[
    \langle a\ast_T b, c\rangle = \sum_{\beta\in H_2^\text{eff}(X)} GW_{0,3}^{X,\beta}(a,b,c) q^\beta
    \]
    for all $a,b,c\in H_T^*(X;\mathbb{Q})$.
    The product is then extended $\Lambda$-bilinearly to $QH_T^*(X)$.
    The \emph{equivariant (small) quantum cohomology ring} is $(QH_T^*(X), \ast_T)$.
\end{defn}
\begin{rem}
    Using the formula for  $\text{vdim}(\overline{\mathcal{M}}_{0,3}(X;\beta))$ one easily checks that $\ast_T$ is a graded product with respect to the above grading of $q^\beta$.
    It is also commutative\footnote{
        Note that in our algebraic and symplectic setting, all cohomology classes have even (real) degrees, see Lemma \ref{lem:torsion_free}.
    }
    and associative with unit $1\in H_T^*(X;\mathbb{Q})$.
\end{rem}
%\begin{rem}
%    The \emph{big} quantum product incorporates Gromov--%Witten invariants with arbitrarily many marked %points. It depends on a parameter in $H_T^*(X)$.\giosue{what does it mean?}
%\end{rem}

Examples of some quantum cohomology rings may be found in \cite[\S8]{Cox_Katz_1999} and \cite[\S11]{McDuff_Salamon_2012}.
In Section~\ref{sec:applications_and_examples} we will use our results to give some new examples of quantum products and Gromov--Witten invariants.

\subsection{Localization formula}
\label{sec:localization_formula}

An analogue for the equivariant localization formula (Theorem \ref{thm:ABBV}) is available for the moduli space of stable maps to calculate (equivariant) Gromov--Witten invariants, see \cite{Kontsevich_1995,Manin_1995,Graber_Pandharipande_1999}.
For algebraic GKM spaces, \cite{LS17} worked out an explicit formula in terms of the GKM graph that requires knowledge of the compatible connection induced by the GKM space and the relations among the curve classes $[C_e]\in H_2(X)$ for $e\in E(G)$.
Although their formula works in arbitrary genus, we will restrict notation to the genus zero case in this paper.
In this Section, we recall their formula and set up the required notation:

% \textbf{Notation.}
\begin{itemize}
    \item $X$ is an $m$-dimensional algebraic GKM space acted on by a torus $T$.
    \item $G$ is the GKM graph of $X$ and $\alpha$ is the axial function.
    \item $\{\nabla_{(e,v)} : (e,v)\in E(G)^\pm\}$ is the compatible connection induced by $X$.
    \item For any  $v\in V(G)$, $\alpha(v):=\prod_{\epsilon\in F(G)_v} \alpha(\epsilon)$.
    %\item $h(\sigma,g=0)=1/w(\sigma)$
    \item Let $e$ be an edge, and let $\epsilon\in F(G)_v$ be one of the two flags defined by $e$. Then for any positive integer $d$,
    \begin{equation}
        \label{eq:h}
        h(e, d) := \frac{(-1)^d d^{2d}}{(d!)^2\alpha(\epsilon)^{2d}}
        \prod_{i=1}^{m-1} b\left(\frac{\alpha(\epsilon)}{d}, \alpha(\epsilon_i), da_i\right)
    \end{equation}
    where
    \[
    b(u,w,a):=\begin{cases}
        \prod_{j=0}^a (w-ju)^{-1} & a\ge 0 \\
        \prod_{j=1}^{-a-1}(w+ju) & a < 0
    \end{cases}
    \]
    and where $F(G)_v = \{\epsilon_1,\dots,\epsilon_{m-1},\epsilon\}$.
    Here, $a_1,\dots,a_{m-1}$ are the integers associated to 
    % the connection function 
    $\nabla_{\epsilon}$ as in Definition~\ref{def:connection}, that is for $1\le i\le m-1$ we have $\alpha(\nabla_{\epsilon} \epsilon_i)=\alpha(\epsilon_i)-a_i\alpha(\epsilon)$. It can easily be seen that $h(e,d)$ does not depend on the choice of the flag $\epsilon$.
    % Definition \ref{def:connection}~(\ref{def:connection:as}), corresponding to $e_1,\dots,e_{r-1}$, respectively.
    % Note that $h(e,d)$ is defined using the \emph{geometric} choice of compatible connection on $G$, coming from the splitting of $T_X|_{C_e}\cong \mathcal{O}(a_1)\oplus\dots\oplus\mathcal{O}(a_r)$ discussed in Section \ref{sec:connections}.

    %$h(\epsilon,d)\in FF(H_T^*)$ only depends on $d$, $w(\epsilon,\sigma)$, $w(\epsilon,\sigma')$ and $a_i$, which only depends on $\epsilon$.
\end{itemize}

The connected components of the fixed locus of $T$ acting on $\overline{\mathcal{M}}_{0,n}(X,\beta)$ correspond bijectively to decorated trees $\overrightarrow{\Gamma}$ consisting of: %\in G_n(X,\beta)
\begin{itemize}
    \item A tree $\Gamma$.
    \item A map of graphs $\overrightarrow{f}\colon \Gamma \rightarrow G$, that is, a map of the sets of vertices, flags and edges of $\Gamma$ and $G$ satisfying obvious compatibility conditions. 
    % \item A \emph{label map} $\overrightarrow{f}\colon V(\Gamma)\sqcup E(\Gamma)\rightarrow V(G)\sqcup E(G)$ labeling edges (respectively vertices) of $\Gamma$ by edges (respectively vertices) of $G$, defining a graph \daniel{\st{map} homomorphism}.\giosue{it is incomplete} \daniel{Why?}\giosue{you leave open the possibility that $\overrightarrow{f}(e)\in V(G)$. LS17 is clear: $\overrightarrow{f}$ sends edges to edges and vertex to vertex.}
    % \daniel{Saying that $\overrightarrow{f}$ defines a graph homomorphism makes it clear (according to Wikipedia's definition of graph homomorphism).}
    \item A \emph{degree map} $\overrightarrow{d}\colon E(\Gamma)\rightarrow \mathbb{Z}_{> 0}$ s.t. $\sum_{e\in E(\Gamma)} \overrightarrow{d}(e)\left[C_{\overrightarrow{f}(e)}\right]=\beta$.
    \item A \emph{marking map} $\overrightarrow{s}\colon \{1,\dots,n\}\rightarrow V(\Gamma)$.
\end{itemize}
Let $\Gamma_n(X, \beta)$ be the set of decorated trees $\overrightarrow{\Gamma}$ as above.
% $\Gamma\in \Gamma_n(X,\beta)$, w
We use the following notation.
\begin{itemize}
    % \item Let $E(\Gamma)_v$ be the edges of $\Gamma$ at \giosue{what does it mean at?} $v\in V(\Gamma)$.
    % \daniel{$E(\Gamma)_v:=\{e \in E(\Gamma)^\mathrm{oriented}:\mathrm{src}(e)=v\}$} \giosue{I think we need to revise all notation. $E(\Gamma)_v$ is a specific subset of flags.}
    \item Let $S_v:=\overrightarrow{s}^{-1}(v)$ be the markings at $v\in V(\Gamma)$, and let $n_v:=|S_v|$.
    \item For a flag $(e,v) \in E(\Gamma)^\pm$, $$\alpha_{(e,v)}:= \frac{\alpha(\overrightarrow{f}(e,v))}{\overrightarrow{d}(e)}.$$ 
    % \item For a flag $(e,v)$, $\alpha_{(e,v)}:= \alpha\left(\overrightarrow{f}(e,v), \overrightarrow{f}(v)\right)/\overrightarrow{d}(e)$ For $e\in E(\Gamma)$ and $v\in V(\Gamma)$ with $v\in e$, let $\alpha_{(e,v)}:= \alpha\left(\overrightarrow{f}(e), \overrightarrow{f}(v)\right)/\overrightarrow{d}(e)$.
    % \giosue{what is $\alpha(\cdot, \cdot)$?}
    % \daniel{Sorry, this should be $\alpha(\overrightarrow{f}(e))/\overrightarrow{d}(e)$ where $e$ is oriented so that $\mathrm{src}(e)=v$}
    \item Let $\text{Aut}(\overrightarrow{\Gamma})$ be the automorphism group of the decorated graph $\overrightarrow{\Gamma}$.
\end{itemize}

The main result of \cite[Theorem~4.7]{LS17}, specialized to genus zero and reformulated for our purposes, is:
\begin{thm}\label{thm:LS17}
    For $\gamma_i\in H_T^*(X)$ we have
    $$GW_{0,n}^{X,\beta}(\gamma_1,\dots,\gamma_n) = \sum_{\overrightarrow{\Gamma}\in \Gamma_n(X,\beta)}GW_{\overrightarrow{\Gamma}}(\gamma_1,\dots,\gamma_n),$$ 
    where
    \begin{multline*}
        GW_{\overrightarrow{\Gamma}}(\gamma_1,\dots,\gamma_n) :=\frac{1}{|\mathrm{Aut}(\overrightarrow{\Gamma})|}
        \prod_{e\in E(\Gamma)} \frac{h(\overrightarrow{f}(e),\overrightarrow{d}(e))}{\overrightarrow{d}(e)}\\
         \prod_{v\in V(\Gamma)}
         \alpha(\overrightarrow{f}(v))^{\mathrm{val}(v)-1}
        \prod_{i\in S_v} \gamma_i|_{\overrightarrow{f}(v)}
        \prod_{e\in E(\Gamma)_v}\alpha_{(e,v)}^{-1}
        \left(
            \sum_{e\in E(\Gamma)_v } \alpha_{(e,v)}^{-1}
        \right)^{n_v+\mathrm{val}(v)-3}.
    \end{multline*}
    % \begin{eqnarray}
    %     GW_{\overrightarrow{\Gamma}}(\gamma_1,\dots,\gamma_n) &=& 
    %     \frac{1}{|\text{Aut}(\overrightarrow{\Gamma})|}
    %     \prod_{e\in E(\Gamma)} \frac{h(\epsilon_e,d_e)}{d_e}\nonumber\\
    %     &&\prod_{v\in V(\Gamma)}\left(w(\sigma_v)^{\text{val}(v)-1}
    %     \prod_{i\in S_v} i_{\sigma_v}^*\gamma_i
    %     \prod_{e\in E_v}w_{(e,v)}^{-1}
    %     \left(
    %         \sum_{e\in E_v} w_{(e,v)}^{-1}
    %     \right)^{n_v+\text{val}(v)-3}
    %     \right)\nonumber
    % \end{eqnarray}
\end{thm}
\begin{proof}
    For $v\in V(\Gamma)$ and $e\in E(\Gamma)_v$, let
    \[
    \psi_{(e,v)}:=c_1(\mathbb{L}_{(e,v)})\in H^2(\overline{\mathcal{M}}_{0,E(\Gamma)_v\cup S_v}),
    \]
    where $\mathbb{L}_{(e,v)}$ is the line bundle over $\overline{\mathcal{M}}_{0,E(\Gamma)_v\cup S_v}$ whose fiber at a moduli point $[C_v,x_v]$ is the cotangent bundle of $C_v$ at the marked point corresponding to $(e,v)$.
    To obtain the desired formula from \cite[Theorem~4.7]{LS17}, set $g=0$ and use
    \[  \int_{\overline{\mathcal{M}}_{0,E(\Gamma)_v\cup S_v}}\frac{1}{\prod_{e\in E(\Gamma)_v}\alpha_{(e,v)}-\psi_{(e,v)}}
    =
    \prod_{e\in E(\Gamma)_v}\alpha_{(e,v)}^{-1}
        \left(
            \sum_{e\in E(\Gamma)_v} \alpha_{(e,v)}^{-1}
        \right)^{n_v+\text{val}(v)-3}
    \]
    which is obtained by expressing the left hand fraction as power series in $\psi_{(e,v)}/ \alpha_{(e,v)}$, using $H_T^*$-linearity of the equivariant integral, and 
    \[ \int_{\overline{\mathcal{M}}_{0,n}}\psi_1^{a_1}\cdots\psi_n^{a_n}=\frac{(n-3)!}{a_1!\dots a_n!}
    \]
    (cf. \cite[Lemma 4.3]{LS17}).
\end{proof}
The following theorem is proved in \cite{Hirschi_2023}.
\begin{thm}\label{thm:Hirschi}
    The formula in Theorem~\ref{thm:LS17} (and its version for arbitrary genus)
    also holds for Gromov--Witten invariants of compact Hamiltonian GKM spaces, defined via global Kuranishi charts.
\end{thm}

\subsection{Independence of the connection}
The formula stated in Theorem~\ref{thm:LS17} appears to depend not only on the underlying GKM graph $G$, but also on the compatible connection on $G$ induced by $X$ (cf. Section~\ref{sec:connections}).
This dependence arises from the presence of the $h$-factors in each individual summand.
% I had to reformulate this, as I found it very hard to read before.
A priori, this could complicate the computation of Gromov--Witten invariants, as there may be two GKM spaces that share the same abstract GKM graph but induce different compatible connections (cf. Remark~\ref{rem:computing_conn}), possibly leading to different Gromov--Witten invariants.
The implication of the following theorem is that we do not need to compute the induced connection, because every compatible connection 
leads to the same computational result.

% I need to reformulate this more smoothly.
% \daniel{We may have two GKM spaces that share the same abstract GKM graph, but induce different compatible connections. 
% Even in this case, the Gromov--Witten invariants are the same because of Theorem~\ref{thm:con_indep}. This has strong implications, because as we anticipated in Remark~\ref{rem:computing_conn}, the computation of a connection induced by a space $X$ could be a difficult problem.} The implication of the following theorem is that we do not need to compute such a connection, because every compatible connection 
% leads to the same computational result.

%To the best of the authors’ knowledge, there exists no general method for determining a GKM space $X$ inducing an abstract GKM graph with connection. Nonetheless, this difficulty does not obstruct the computation of equivariant Gromov--Witten invariants of GKM spaces because of the following.

%\giosue{A priori, the formula of Theorem \ref{thm:LS17} depends not only on the GKM graph $G$, but also\footnote{because of the $h$-factors in every single summand} on knowing the compatible connection on $G$ coming from the geometry of $X$ (see Section \ref{sec:connections}).
%In fact, the authors are not aware of a general way of determining the geometric choice of compatible connection for any given abstract GKM graph.
%Fortunately, it turns out that we do not need to solve this problem in order to compute equivariant Gromov--Witten invariants of GKM spaces.}

\begin{thm}\label{thm:con_indep}
    The formula in Theorem \ref{thm:LS17} (and its version for arbitrary genus) is independent of the choice of compatible connection on the GKM graph $G$.
\end{thm}

\begin{proof}
    Let $e=\{p,q\}$ be an edge of $G$, $\epsilon=(e,p)\in E(G)^\pm$ and $d\ge 1$.
    Let $\epsilon_1,\dots,\epsilon_{m-1}$ and $a_1,\dots,a_{m-1}$ be as in Section~\ref{sec:localization_formula}.
    It suffices to show that the factor 
    \[
        h(e, d) = \frac{(-1)^d d^{2d}}{(d!)^2\alpha(\epsilon)^{2d}}
        \prod_{i=1}^{m-1} b\left(\frac{\alpha(\epsilon)}{d}, \alpha(\epsilon_i), da_i\right)
    \]
    from \cite[Lemma~4.5]{LS17} does not depend on the choice of compatible connection, since this is the only point where the connection enters the formula.
    By definition, we have
    \[
    b(u,w,a)=\begin{cases}
        \prod_{j=0}^a (w-ju)^{-1} & a\ge 0 \\
        \prod_{j=1}^{-a-1}(w+ju) & a < 0.
    \end{cases}
    \]
    As noted in \cite[Example~19]{Liu12}, the factors of $b(u,w,a)$ are given by expressing
    \begin{equation}\label{eqn:exp_frac_sum}
    \frac{e^{w}}{1-e^{-u}} + \frac{e^{w-au}}{1-e^{u}}
    \end{equation}
    as $\sum_j c_je^{f_j}$ where $c_j\in\mathbb{Z}$ and $f_j\in \mathfrak{t}^\ast$.
    Then 
    $b(u,v,w) = \prod_j f_j^{c_j}$.
    Note that the $(f_j,c_j)$ are unique by linear independence of characters.
    If $\nabla_\epsilon(\epsilon_i)=\epsilon_i'$, then Equation~\eqref{eqn:exp_frac_sum} for $b\left(\frac{\alpha(\epsilon)}{d}, \alpha(\epsilon_i), da_i\right)$ is
    \[
    \frac{e^{\alpha(\epsilon_i)}}{1 - e^{-\alpha(\epsilon)/d}}
    + \frac{e^{\alpha(\epsilon_i')}}{1-e^{\alpha(\epsilon)/d}}.
    \]
    Summing this over all $i\in\{1,\dots,m-1\}$ gives
    \[
    \sum_{i=1}^{m-1} \frac{e^{\alpha(\epsilon_i)}}{1 - e^{-\alpha(\epsilon)/d}}
    + \sum_{i=1}^{m-1} \frac{e^{\alpha(\epsilon_i')}}{1-e^{\alpha(\epsilon)/d}}.
    \]
    This expression no longer depends on the bijection 
    $$\nabla_\epsilon\colon\{\epsilon_1,\dots,\epsilon_{m-1}\}\longrightarrow\{\epsilon_1',\dots,\epsilon_{m-1}'\}.$$
    As before, writing this as $\sum_j c_je^{f_j}$ gives
    $$\prod_{i=1}^{m-1} b\left(\frac{\alpha(\epsilon)}{d}, \alpha(\epsilon_i), da_i\right)=\prod_j f_j^{c_j},$$
    which is therefore independent of the connection.
\end{proof}

\subsection{A connection-free expression.}
The expression
$$\prod_{i=1}^{m-1} b\left(\frac{\alpha(\epsilon)}{d}, \alpha(\epsilon_i), da_i\right)$$
can be reformulated in a fully explicit manner that does not depend on the connection. 
% We unravel the above to obtain an expression for $\prod_{i=1}^{r-1} b\left(\frac{\alpha(e)}{d}, \alpha(e_i), da_i\right)$ which does not require any choice of connection on $G$.
First, we partition the sets $E(G)_{p}\setminus\{\epsilon\}$ and $E(G)_{q}\setminus\{\overline{\epsilon}\}$ as
\[
E(G)_{p}\setminus\{\epsilon\} = E_1\sqcup\dots\sqcup E_k,\hspace{10mm}
E(G)_{q}\setminus\{\overline{\epsilon}\} = E_1'\sqcup\dots\sqcup E_k'
\]
where each $|E_i|=|E_i'|>0$ and for any $\tilde{\epsilon},\tilde{\epsilon}'\in E_i\cup E_i'$ we have $\alpha(\tilde{\epsilon})-\alpha(\tilde{\epsilon}')\in \mathbb{Z}\cdot\alpha(\epsilon)$.
The partition is made unique up to labeling by requiring the number of parts to be minimal.
Such a partition exists as soon as the GKM graph admits a compatible connection $\nabla_\epsilon$ at $\epsilon$, but it is independent of the choice of connection.
(Choosing a connection amounts to choosing a bijection $E_i\rightarrow E_i'$ for each $i$.)
We then have
\[
    \prod_{i=1}^{m-1} b\left(\frac{\alpha(\epsilon)}{d}, \alpha(\epsilon_i), da_i\right)
    = \prod_{i=1}^k B_i
\]
where the rational functions $B_1,\dots,B_k$ in $t_1,\dots,t_r$ are defined as follows.
Fix $i\in\{1,\dots,k\}$.
By definition, the weights $E_i\cup E_i'$ all lie on the same line $L_i$ in $\mathfrak{t}^*$ with slope $\alpha(\epsilon)$.
Let $\Lambda_i\subset L_i$ be the lattice $E_i+\frac{1}{d}\mathbb{Z}\cdot \alpha(\epsilon)$.
Orient $L_i$ such that \textit{to the right} means in positive $\alpha(\epsilon)$-direction.
For any $f\in \Lambda_i$, define
\begin{align*}
    \iota(f) &:= |\{p\in E_i: p\text{ is strictly to the left of }f\}|,\\
    \iota'(f) &:= |\{p\in E_i': p\text{ is strictly to the right of }f\}|.
\end{align*}
$B_i$ is then defined by the (finite) product
\[
B_i = \prod_{f\in \Lambda_i} f^{\iota(f)+\iota'(f) - |E_i|}.
\]
Theorem \ref{thm:con_indep} and Proposition \ref{prop:curve_classes} give the following theorem.

\begin{thm}\label{cor:GKM_determines_GW}
    The equivariant Gromov--Witten invariants of an 
    % algebraic or Hamiltonian 
    Hamiltonian or algebraic
    GKM space can be computed purely in terms of the GKM graph, without knowledge of the compatible connection induced by the space.
\end{thm}

\section{A computational tool}
\label{sec:tool}
A remarkable feature of algebraic and Hamiltonian GKM spaces is that many of their key invariants—including (equivariant) Gromov--Witten invariants—are determined purely by their GKM graph.
This makes them especially amenable to explicit computations, provided one has the appropriate computational framework.

For this purpose, we have developed the package {\pkg} in \texttt{Julia} (see \cite{bezanson2017julia}), designed for efficient and versatile computations in this setting.

The package is built on the computer algebra system \texttt{OSCAR} \cite{MR4886702}, and significantly extends a series of earlier packages by the second author \cite{MR4383164,Mur24,MR4786714}, incorporating new algorithms and a broader computational scope.
Table~\ref{tab:package_features} summarizes its current capabilities.
To support positive genus computations, we used the package introduced in \cite{Yang2010} to compute the Hodge integrals pre-loaded in {\pkg}, and \cite{McKay,MR3131381} for the generation of the graphs.

In the remainder of this section we give a concise, hands-on introduction to {\pkg}, aimed at getting the reader up and running quickly.
Full technical documentation, including installation instructions and advanced examples, is available at
\begin{center}
\pkgwebsite
\end{center}

First, we load \texttt{OSCAR} and {\pkg} in the Julia command line, as they are already installed in our machine. After that, we define $G$ to be the GKM graph of the Grassmannian $G(2,4)$.
\begin{shaded}
\begin{MyVerb}[fontsize=\small]
\textcolor{green}{julia>} using Oscar, GKMtools
\textcolor{green}{julia>} G = grassmannian(GKM_graph, 2, 4)
GKM graph with 6 nodes, valency 4 and axial function:
13 -> 12 => (0, -1, 1, 0)
14 -> 12 => (0, -1, 0, 1)
14 -> 13 => (0, 0, -1, 1)
23 -> 12 => (-1, 0, 1, 0)
23 -> 13 => (-1, 1, 0, 0)
24 -> 12 => (-1, 0, 0, 1)
24 -> 14 => (-1, 1, 0, 0)
24 -> 23 => (0, 0, -1, 1)
34 -> 13 => (-1, 0, 0, 1)
34 -> 14 => (-1, 0, 1, 0)
34 -> 23 => (0, -1, 0, 1)
34 -> 24 => (0, -1, 1, 0)
\end{MyVerb}
% \begin{Verbatim}[fontsize=\small]
% julia> using Oscar, GKMtools
% julia> G = grassmannian(GKM_graph, 2, 4)
% GKM graph with 6 nodes, valency 4 and axial function:
% 13 -> 12 => (0, -1, 1, 0)
% 14 -> 12 => (0, -1, 0, 1)
% 14 -> 13 => (0, 0, -1, 1)
% 23 -> 12 => (-1, 0, 1, 0)
% 23 -> 13 => (-1, 1, 0, 0)
% 24 -> 12 => (-1, 0, 0, 1)
% 24 -> 14 => (-1, 1, 0, 0)
% 24 -> 23 => (0, 0, -1, 1)
% 34 -> 13 => (-1, 0, 0, 1)
% 34 -> 14 => (-1, 0, 1, 0)
% 34 -> 23 => (0, -1, 0, 1)
% 34 -> 24 => (0, -1, 1, 0)
% \end{Verbatim}
\end{shaded}

This command defines $G$ and it shows all edges of $G$ with the value of the axial function for each (oriented) edge. The symbols $\{13,12,14,\ldots\}$ are the vertices of $G$, each one represents a certain $(\mathbb{C}^\times)^4$-invariant flag of $\mathbb{C}^4$. Now, we define $H$ to be the GKM graph of the product $G(2,4)\times G(2,4)$.
\begin{shaded}
\begin{MyVerb}[fontsize=\small]
\textcolor{green}{julia>} H = G*G
GKM graph with 36 nodes, valency 8 and axial function:
13,12 -> 12,12 => (0, -1, 1, 0, 0, 0, 0, 0)
14,12 -> 12,12 => (0, -1, 0, 1, 0, 0, 0, 0)
[output partly omitted]
\end{MyVerb}
\end{shaded}
In order to define the curve classes of $H$, we call the following function:
\begin{shaded}
\begin{MyVerb}[fontsize=\small]
\textcolor{green}{julia>} print_curve_classes(H)
34,12 -> 24,12: (1, 0), Chern number: 4
12,13 -> 12,12: (0, 1), Chern number: 4
[output partly omitted]
\end{MyVerb}
\end{shaded}
In this case, $H_2(H, \mathbb{Z})\cong \mathbb{Z}^2$. So in order to define the class $b$ corresponding to $(1,1)$, we call
\begin{shaded}
\begin{MyVerb}[fontsize=\small]
\textcolor{green}{julia>} b1 = curve_class(H, "34,12", "24,12");
\textcolor{green}{julia>} b2 = curve_class(H, "12,13", "12,12");
\textcolor{green}{julia>} b = b1 + b2
(1, 1)
\end{MyVerb}
\end{shaded}
Finally, in order to compute the Gromov--Witten invariant
$$\int_{[\overline{\mathcal{M}}_{0,2}(H;b)]^\text{vir}} \ev_1^*([\mathrm{pt}])\smile\ev_2^*(\mathrm{PD}(b))=0,$$
where $\mathrm{PD}(b)$ is the Poincar{\'e} dual of $b$, we use the following code:
\begin{shaded}
\begin{MyVerb}[fontsize=\small]
\textcolor{green}{julia>} e1 = ev(1, point_class(H, 1));
\textcolor{green}{julia>} PD = poincare_dual(gkm_subgraph_from_vertices(H, ["34,12", "24,12"]));
\textcolor{green}{julia>} PD += poincare_dual(gkm_subgraph_from_vertices(H, ["12,13","12,12"]));
\textcolor{green}{julia>} gromov_witten(H, b, 2, e1 * ev(2, PD); fast_mode = true)
0
\end{MyVerb}
\end{shaded}

Let us briefly see another example. Let $F$ be the GKM graph of the homogeneous variety $G_2/B$. The vertices of $F$ are labeled by elements of the Weyl group of $G_2$, generated by the reflections $\{s_1,s_2\}$ with respect to the two roots. Let $\mathrm{id}$ be the neutral element of the Weyl group. Let $c_1$ and $c_2$ be the $T$-invariant curves from the vertex $\mathrm{id}$ to, respectively, the vertices $s_1$ and $s_2$. Thus,
$$\int_{[\overline{\mathcal{M}}_{0,1}(F;c_1)]^\text{vir}} \ev_1^*([\mathrm{pt}])=
\int_{[\overline{\mathcal{M}}_{0,1}(F;c_2)]^\text{vir}} \ev_1^*([\mathrm{pt}])=1$$
can be obtained as:
\begin{shaded}
\begin{MyVerb}[fontsize=\small]
\textcolor{green}{julia>} F = generalized_gkm_flag(root_system(:G, 2));
\textcolor{green}{julia>} c1 = curve_class(F, "id", "s1");
\textcolor{green}{julia>} c2 = curve_class(F, "id", "s2");
\textcolor{green}{julia>} gromov_witten(F, c1, 1, ev(1, point_class(F, 1)))
1
\textcolor{green}{julia>} gromov_witten(F, c2, 1, ev(1, point_class(F, 1)))
1
\end{MyVerb}
\end{shaded}
We used only the curves $c_1$ and $c_2$ in this example. But all $T$-invariant curves of $F$ are supported using the function \verb|curve_class|.

Finally, we provide a computation of a positive genus Gromov--Witten invariant. 
Let $F$ be a twisted flag manifold (see Section~\ref{sec:twisted_flag}) and $\beta$ the  unique primitive curve class with $\int_\beta c_1(T_F)=0$ (cf. Figure~\ref{fig:F3_twisted}).
The following code computes the invariant
$$\int_{[\overline{\mathcal{M}}_{1,0}(F;d\beta)]} 1$$
% \[
% \int_{[\overline{\mathcal{M}}_{1,0}(F,\beta)]} 1=\frac{1}{12}
% \]
for $d=1,2,3$.

\begin{shaded}
\begin{MyVerb}[fontsize=\small]
\textcolor{green}{julia>} F = gkm_3d_twisted_flag();
\textcolor{green}{julia>} beta = curve_class(F, Edge(3, 4)); # unique edge with C_1 = 0.
\textcolor{green}{julia>} [gromov_witten(F, d*beta, 0, class_one(); g=1) for d in 1:3]
1//12
-1//24
-29//36
\end{MyVerb}
\end{shaded}
This agrees with the calculation on the appropriately linearized equivariant Calabi--Yau vector bundle $\mathcal{O}_{\mathbb{P}^1}(-3)\oplus\mathcal{O}_{\mathbb{P}^1}(1)$, which models the neighborhood of \verb|Edge(3, 4)| on the GKM graph of $F$.

\begin{table}[htb]
    \centering
    \begin{tabular}{|p{3cm}|p{4cm}|p{4cm}|}
        \hline
        \textbf{Construction}
        % \textbf{Mathematical concept}
        & \textbf{Supported operations}
        & \textbf{Examples} %\textbf{Supported standard examples}
        \\\hline\hline
        \parbox{3cm}{\textbf{GKM graphs}}
        & \parbox{4cm}{
            \begin{itemize}[leftmargin=*]
                \item Products. \item Blowups.
                \item Vector bundles. \item Projectivization thereof.%of vb%ector bundles
                \item Subgraphs. 
                \item Betti numbers.
                % \item calculating Betti numbers
            \end{itemize}
        }
        & \parbox{4cm}{
            \begin{itemize}[leftmargin=*]
                \item Smooth projective toric varieties.
                \item (partial) flag varieties$^{(1)}$
                \item Homogeneous spaces.%, that is $G/P$
                \item Smooth Schubert varieties in $G/P$.
                \item 3D GKM fibrations \cite{GKZ20}.
            \end{itemize}
        }
        \\\hline
        \parbox{3cm}{\textbf{Compatible\,\,\,\, Connections}}
        &
        \parbox{4cm}{
        \begin{itemize}[leftmargin=*]
            \item Calculation of the unique one for 3-ind. GKM graphs.
            \item Compatibility check.
        \end{itemize}
        }
        &
        \parbox{4cm}{
        \begin{itemize}[leftmargin=*]
            \item 
             Compatible connection from splitting of $T_X$ in all standard examples.
        \end{itemize}
        }
        \\\hline
        \parbox{3cm}{
        \textbf{Equivariant\,\,\,\, Cohomology}
        }
        & \parbox{4cm}{
            \begin{itemize}[leftmargin=*]
                \item arithmetic: $+, -, \times$.
                \item Equivariant integration.
            \end{itemize}
        }
        & \parbox{4cm}{
            \begin{itemize}[leftmargin=*]
                \item Chern classes.
                \item Poincar{\'e} duals of GKM subgraphs.
            \end{itemize}
        }
        \\
        \hline
        \parbox{3cm}{
        \textbf{Curve classes}
        }
        &\multicolumn{2}{l|}{
        \parbox{8cm}{
        \begin{itemize}[leftmargin=*]
            \item Kernel of $\mathbb{Z}^{E(G)}\rightarrow H_2(X)$.
            \item Enumeration of multiplicities $E(G)\rightarrow \mathbb{N}_{\ge 0}$ giving a prescribed curve class.
        \end{itemize}
        }}
        \\
        \hline
        \parbox{3cm}{\textbf{Equivariant\,\,\,\, GW invariants}}
        &
        \multicolumn{2}{l|}{
        \parbox{8cm}{
        \begin{itemize}[leftmargin=*]
            \item Calculation from the GKM graph for any $g\ge 0$.
            \item Including $\ev^*(\cdot)$ and $\psi$-classes.
        \end{itemize}
        }}
        \\
        \hline
        \parbox{3cm}{
        \textbf{Equivariant\,\,\,\,  Quantum\,\,\,\,\,\,\,\,\,\,\,\,\,\,  Cohomology}
        }
        &
        \multicolumn{2}{l|}{
        \parbox{8cm}{
        \begin{itemize}[leftmargin=*]
            \item Calculation of $\ast_T$ structure constants in given curve class.
            \item Possible for all relevant curve classes if ${\mathcal{C}_1(e)>0}\,\,\,\forall e\in E(G)$.
            \item Arithmetic with elements of $QH_T^*(X)$.
        \end{itemize}
        }
        }
        \\
        \hline
    \end{tabular}
    \vspace{2mm}
    \caption{Scope of {\pkg} at the time of writing. (1.) For performance reasons, (partial) flag varieties are implemented separately, although they are of course type $A_n$ homogeneous spaces.}
    \label{tab:package_features}
\end{table}

\begin{rem}
    The current stable release of {\pkg} is version~\pkgversion. Its range of applications is expected to expand in future developments. For an up-to-date and comprehensive overview of the package’s capabilities, we refer the reader to the online documentation.
\end{rem}

\section{Applications and examples}
\label{sec:applications_and_examples}

In the remaining part of the paper, we provide some interesting examples that were calculated using the tool from Section~\ref{sec:tool},
% Comment daniel: I'm trying to highlight here a bit more that Section 5 has mathematical value beyond the simple application of our tool.
as well as proofs of patterns that we identified.
We present applications to Hamiltonian realizability in GKM theory, enumeration of plane cubics in $\mathbb{P}^4$, and equivariant quantum Schubert calculus for smooth Schubert varieties.
% In many cases, this allowed us to find a pattern that we prove in this section.

\subsection{Calabi--Yau local models for \texorpdfstring{$\mathbb P^1$}{P1}}
\label{sec:CY_local_models}

A particularly beautiful scenario for computing Gromov--Witten invariants is that of Calabi--Yau threefolds.
In this case, we have
\[
\text{vdim}(\overline{\mathcal{M}}_{g,n}(X;\beta)) = n
\]
so working with $n=0$ marked points implies that $GW_{g,0}^{X,\beta}\in\mathbb{Q}$.
Hence, this should naively be a \emph{``count of genus $g$ curves on $X$ in curve class $\beta$"}.
It is a classical question to ask how
% $GW_{g,0}^{X,\beta}$ 
$$GW_{g,0}^{X,\beta}=\int^T_{[\overline{\mathcal{M}}_{g,0}(X;\beta)]^\text{vir}}1$$
relates to the actual number of such curves \cite[\S7.4]{Cox_Katz_1999}.
An important special case is to consider the contribution of an \emph{isolated}, \emph{rigid}, or \emph{super-rigid} curve (see e.g. \cite{MR1729095_Pandharipande_1999}).

\subsubsection{Background}

Using string theory, Gopakumar and Vafa arrived at the formula
\begin{align*}
    \sum_{\beta\neq 0, g\ge 0} GW_{g,0}^{X,\beta} t^{2g-2}q^\beta
    =
    \sum_{\beta\neq 0, h\ge 0}n_{h,\beta}(X)\sum_{k=1}^\infty \frac{1}{k}\left(2\sin \left(\frac{kt}{2}\right)\right)^{2h-2} q^{k\beta}
\end{align*}
and conjectured that the \emph{BPS-invariants} $n_{h,\beta}(X)$ defined by this equation should be integers (see \cite[\S34]{Hori_Katz_et_al_2003} and references therein).
Furthermore, for any $\beta$ there should exists $h_\beta\in\mathbb{Z}$ such that for all $h\ge h_\beta$, $n_{\beta,h}(X)=0$.
Restricting to the genus zero part gives
\begin{align}
    \label{eqn:GV_genus_zero}
    GW_{0,0}^{X,\beta} = \sum_{\beta = d\gamma} \frac{n_{0,\gamma}(X)}{d^3}.
\end{align}

It is tempting to interpret 
$n_{0,\beta}(X)$ as the number of genus $0$ curves in class $\beta$.
In particular, the degree $d$ covers of a $\mathbb{P}^1$ in class $\beta$ should contribute $1/d^3$ to $GW_{0,0}^{X,d\beta}$.
While this interpretation is false in general (see \cite[\S7.4.3]{Cox_Katz_1999} and \cite[\S29.1.2]{Hori_Katz_et_al_2003}), it is correct for certain spaces like $\mathcal{O}_{\mathbb{P}^1}(-1)\oplus\mathcal{O}_{\mathbb{P}^1}(-1)\rightarrow\mathbb{P}^1$, which is the local model for a rigid $\mathbb{P}^1$ in a Calabi--Yau threefold (see \cite[\S27.5]{Hori_Katz_et_al_2003} and references therein).

Interesting local models to study in this context are rank 2 vector bundles ${L_1\oplus L_2\rightarrow C}$ over a curve $C$.
In this case, Bryan and Pandharipande worked out the theory of \emph{disconnected} equivariant Gromov--Witten invariants, where the domain of a stable map may be disconnected \cite{Bryan_Pandharipande_2008}.
They consider a $T$-action on the total space that is trivial on the base curve $C$.

For symplectic Calabi--Yau $6$-manifolds, the Gopakumar--Vafa conjectures were proved in \cite{MR3739228,doan}.

\begin{rem}\label{rem:BP08}
    In principle, one could try to deduce the results in Section~\ref{sec:CY_local_models} by specializing the generating function in \cite[Corollary~7.2]{Bryan_Pandharipande_2008} appropriately, taking its logarithm, and dealing with the resulting partition combinatorics.
    On the other hand, our proofs do not rely on \cite{Bryan_Pandharipande_2008} and are based directly on the localization formula (Theorem~\ref{thm:LS17}).
\end{rem}

\subsubsection{Our setting}
\label{sec:our_local_setting}

We shall work with $T$-equivariant vector bundles $X\rightarrow \mathbb{P}^1$.
Throughout, we assume that the $T$-action on $\mathbb{P}^1$ has precisely two fixed points.
Our setting differs from that considered in \cite{Bryan_Pandharipande_2008} because we have a non-trivial $T$-action on the base and work with \emph{connected} Gromov--Witten invariants.
The class of the zero section will be denoted by $\mathbf{0}_X$.

% A quasi-projective smooth threefold $X$ is Calabi--Yau if the canonical bundle is trivial. Since we work with the only  We know define 

\begin{defn}\label{def:Calabi}
    Let $X$ be the total space of the $T$-linearized rank $2$ vector bundle $\mathcal{O}_{\mathbb{P}^1}(a_1)\oplus\mathcal{O}_{\mathbb{P}^1}(a_2)\rightarrow \mathbb{P}^1$.
    We say that $X$ is:
    \begin{itemize}
        \item \emph{Calabi--Yau} if $c_1(T_X)=0$,
        \item \emph{equivariantly Calabi--Yau} if $c_1^T(T_X)=0$.
    \end{itemize}
\end{defn}
Equivalently, $X$ is Calabi--Yau if $\det(\Omega_X)=0$, or if $a_1+a_2=-2$. On the other hand, equivariantly Calabi--Yau implies the Calabi--Yau property and is equivalent to the sum of all tangent weights being zero at both fixed points of $\mathbb{P}^1$. 

% \giosue{From now on, we denote by $X_k$ for $k\ge 0$ the $T$-linearized Calabi--Yau rank $2$ vector bundle $\mathcal{O}_{\mathbb{P}^1}(k-1)\oplus\mathcal{O}_{\mathbb{P}^1}(-k-1)$.}
% \daniel{How about the following? I want to prevent the reader worrying "oh, but wich $T$-linearization do they mean?" when we actually mean any that gives a GKM space.}
% \st{From now on, we denote by $X_k$ for $k\ge 0$ the Calabi--Yau rank $2$ vector bundle $\mathcal{O}_{\mathbb{P}^1}(k-1)\oplus\mathcal{O}_{\mathbb{P}^1}(-k-1)$ with any $T$-linearization. }\giosue{but the T-linearization is not already written in the image??}
% \daniel{$t_1, t_2, t_3$ are just names for the axial function values, not definitions. In particular, they will sometimes later satisfy different relations such as $t_1+t_2+t_3=0$ or other relations. So how about the following:}
From now on, we denote by $X_k$ for $k\ge 0$ the $T$-linearized Calabi--Yau rank 2 vector bundle $\mathcal{O}_{\mathbb{P}^1}(k-1)\oplus\mathcal{O}_{\mathbb{P}^1}(-k-1)$ with $T$-weights labeled as in Figure~\ref{fig:localP1}.

% \begin{defn}[\daniel{cf somewhere in BP09}\giosue{In page 3 they say that $CY:=(\det(T_X)=0)$ which is not the same as $c_1(T_X)=0$. In page 34 they "define" equivariantly CY. Maybe some comments saying that in our case where we take $\mathbb P^1$ as base, those definitions coincide.}]
%     Let $\mathcal{O}(a_1)\oplus\mathcal{O}(a_2)\rightarrow \mathbb{CP}^1$ be a $T$-linearized vector bundle with $a_1\ge a_2$ and total space $X$.
%     We say that the bundle is:
%     \begin{itemize}
%         \item \emph{Calabi--Yau} if $a_1+a_2 = -2$. This is equivalent to the total space having vanishing non-equivariant first Chern class $c_1(T_X)$.

%         \item \emph{equivariantly Calabi--Yau} if the equivariant first Chern class $c_1^T(T_X)$ vanishes.
%         This automatically implies the Calabi--Yau property and is equivalent to the sum of all tangent weights being zero at both fixed points of $\mathbb{P}^1$.
%     \end{itemize}
% \end{defn}
% Figure~\ref{fig:localP1} shows our convention for naming the weights of this local model.
% Note that $c_1^T(T_X)=0$ if and only if $t_1+t_2+t_3=0$.

\begin{figure}[htb]
    \centering
    \includegraphics[width=0.5\linewidth]{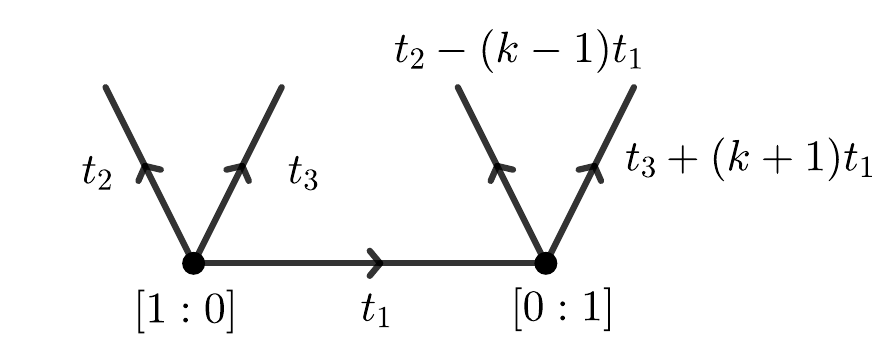}
    \caption{
    The GKM graph of $X_k$.
    Here, $t_1$ is the weight of $T_{\mathbb{P}^1,[1:0]}$, and $t_2,t_3$ are the weights of the fibers at $[1:0]$.
    The weights at $[0:1]$ are automatically determined by the degrees $a_1=k-1$ and $a_2=-k-1$ (see Section~\ref{sec:connections}).}
    \label{fig:localP1}
\end{figure}

We consider this as an algebraic non-projective GKM space whose GKM graph contains a flag for each $T$-invariant fiber (see Figure~\ref{fig:localP1}).
The same localization formula from Theorem~\ref{thm:LS17} then applies to calculate equivariant Gromov--Witten invariants of $X_k$.
However, since $X_k$ is not proper, $\overline{\mathcal{M}}_{g,n}(X_k;\beta)$ is not necessarily proper and hence the equivariant integral is not a priori guaranteed to lie in $H_T^*(\text{pt};\mathbb{Q})$.
Instead, it generally lies in the fraction field of $H_T^*(\text{pt};\mathbb{Q})$, so it is a rational function in the equivariant parameters $t_1,t_2,t_3$.
By the Calabi--Yau threefold assumption, $GW_{0,0}^{X_k,\beta}$ is a rational function of homogeneous total degree zero (see Section~\ref{sec:GW}).

\begin{lem}
\label{lem:polynomiality_condition}
Let $k\ge 1$. Hence
$$
    GW_{0,0}^{X_k,\mathbf{0}_{X_k}}=
    \frac{(t_1+t_3)(2t_1+t_3)\cdots(kt_1+t_3)}{t_2(t_2-t_1)(t_2-2t_1)\cdots(t_2-(k-1)t_1)}
    =\prod_{i=1}^k\frac{it_1+t_3}{t_2 - (i-1)t_1}.
$$
If $k\ge 2$, $GW_{0,0}^{X_k,\mathbf{0}_{X_k}}$ is a constant if and only if
$$t_3 \in \{ -t_1 -t_2,  -kt_1 + t_2\}.$$
%\daniel{For $t_3 = -t_1 + yt_2$, I get $GW_{0,0}^{X_{1},d\cdot\mathbf{0}_{X_1}} =y/d^3$.}
\end{lem}
\begin{proof}
    The expression for the Gromov--Witten invariant follows by applying the localization formula of Theorem~\ref{thm:LS17} to the GKM graph of $X_k$ given in Figure~\ref{fig:localP1}.

    Let $t_3=at_1+bt_2$. Since $X_k$ is $2$-independent, we must have $a\neq 0\neq b$. If the expression
    \[GW_{0,0}^{X_k,\mathbf{0}_{X_k}}=\prod_{i=1}^k\frac{(i+a)t_1+bt_2}{t_2 - (i-1)t_1}\]
    is constant then it must be equal to $b^k$, simply by taking $t_1=0$. Thus, dividing by $b^k$ we have
    \[\prod_{i=1}^k ((i+a)t_1+bt_2) = \prod_{i=1}^k (bt_2 +b (1-i)t_1).\]
    This is possible if and only if
    $$\{i+a:1\le i \le k\} = \{0, -b, -2b, \ldots,(1-k)b\}.$$
    Clearly, $-k\le a \le -1$ and $b=\pm 1$. It can easily be seen that the only two possible solutions are $a=b=-1$ and $a=-k, b=1$. 
\end{proof}

Thus, for $k\ge 1$ we have the following cases in which we can hope for polynomial Gromov--Witten invariants of $X_k$:
%\begin{enumerate}
%    \item the \emph{equivariantly Calabi--Yau case} %$t_3=-t_1-t_2$ corresponding to %$c_1^T(T_{X_k})=0$,
%    \item the \emph{non--equivariantly Calabi--Yau %case} case $t_3=t_2-kt_1$ corresponding to %$c_1^T(T_{X_k})=(1-k)t_1 + 2t_2$,
%    \item for $k=1$ the case $t_3=-t_1+yt_2$ where %$y$ is any non-zero integer, corresponding to %$c_1^T(T_{X_k})=(y+1)y_2$
%    %Non-zero because otherwise the GKM graph is %not 2-independent.
%\end{enumerate}
\begin{equation}\label{eqn:polynomial_cases}
    \begin{cases}
        t_3 = -t_1 - t_2 & \Longleftrightarrow\,\,\,\,  c_1^T(T_{X_k})=0%\hspace{3mm}\text{(the equivariantly Calabi-Yau case)}
        \\
        t_3 = -kt_1+t_2 & \Longleftrightarrow\,\,\,\,  c_1^T(T_{X_k})=(1-k)t_1 + 2t_2
        \\
        t_3=-t_1+yt_2 & \Longleftrightarrow\,\,\,\,  c_1^T(T_{X_1}) = (1+y)t_2 \hspace{3mm} %\text{(only for }
        (k=1\text{, any }y\in\mathbb{Z}\setminus\{0\})
    \end{cases}
\end{equation}
The first case is the \emph{equivariantly Calabi--Yau} case (see Definition~\ref{def:Calabi}).

%Thus, we have the following cases in which we can %hope for polynomial Gromov--Witten invariants: the %equivariantly Calabi--Yau case, \st{and the %\emph{"other"}}\giosue{non-equivariantly CY} %\st{case with $t_3 = t_2 - kt_1$}.
%\daniel{the case $t_3=t_2 - kt_1$, and the case %$k=1$ and $t_3=-t_1+yt_2$ for some integer $y$.}

\begin{lem}\label{lem:polarity}
    % For any $d\ge 1$, t
    The $T$-equivariant Gromov--Witten invariant $GW_{0,0}^{X_k,d\mathbf{0}_{X_k}}$
    does not have $t_1$ as factor in its denominator.
    % , as a meromorphic function, has no pole on $t_1=0$.
\end{lem}

\begin{proof}
    Let $T'\subset T$ be the subtorus with Lie algebra $\ker(t_1)\subset\mathfrak{t}$.
    Then $T'$ acts trivially on the zero section of $X_k$ and scales the summands $\mathcal{O}_{\mathbb{P}^1}(k-1),\mathcal{O}_{\mathbb{P}^1}(-k-1)$ by weights $t_2|_{\ker(t_1)}, t_3|_{\ker(t_1)}$, respectively.
    By the 2-independence of GKM graphs, $t_2|_{\ker(t_1)},t_3|_{\ker(t_1)}\neq 0$ and hence
    $X_k^{T'}$ is precisely the zero section.
    Thus, the $T'$-equivariant Gromov--Witten invariant $GW_{0,0}^{X_k,d\mathbf{0}_{X_k}}$ is well-defined by virtual localization (cf. \cite[\S3.5.2]{LS17}).
    Let $I_T$ (resp., $I_{T'}$) be the $T$-equivariant (resp., $T'$-equivariant) Gromov--Witten invariant $GW_{0,0}^{X_k,d\mathbf{0}_{X_k}}$.
    The forgetful map $H_T^*\rightarrow H^*_{T'}$ has as kernel the linear span of $t_1$ and sends $I_T$ to $T_{T'}$.
    Thus, $I_T$ cannot have $t_1$ as factor in the denominator.
    
    % Let $T'\subset T$ be the pointwise stabilizer of the zero section $\mathbb{P}^1\subset X$.
    % Since $T$ acts on $\mathbb{P}^1$ with primitive weight, $T'$ is a subtorus of $T$.
    % The inclusion $T'\hookrightarrow T$ induces a map $H_T^*\rightarrow H_{T'}^*$ 
    % % \st{sending $t_1\mapsto 0$ and $t_2,t_3\not\mapsto 0$} 
    % corresponding to the restriction to the hyperplane $t_1=0$.
    % Since $X^{T'}=\mathbb{P}^1$ is still compact, one may use virtual localization to calculate the $T'$-equivariant $GW_{0,0}^{X_k,d\mathbf{0}_{X_k}}$, which is a homogeneous function of degree zero in the fraction field of $H_{T'}^*$.
    % It is obtained by applying $H_T^*\rightarrow H_{T'}^*$ to the $T$-equivariant $GW_{0,0}^{X_k,d\mathbf{0}_{X_k}}$.
    % Hence, setting $t_1=0$ in the $T$-equivariant $GW_{0,0}^{X_k,d\mathbf{0}_{X_k}}$ must yield a well-defined element of the fraction field of $H_{T'}^*$.
\end{proof}

The following theorem shows that all cases from \eqref{eqn:polynomial_cases} indeed give rise to Gromov--Witten invariants that are constant in the equivariant parameters.
Note that for $k=0$ we have $GW_{0,0}^{X_k,d\mathbf{0}_{X_k}}=1/d^3$  (see \cite{Manin_1995} and \cite[\S27.5]{Hori_Katz_et_al_2003}), so we focus on the case $k>0$.
A completely different proof of Theorem~\ref{thm:gw_equivariant_CY}~(1) below appeared independently in \cite[Theorem~3.2]{vangarrel2025}, as we explained in the Introduction.
% \pagebreak
% like that? yes! :) thx!
% I'm submitting now. I sent you a draft of a possible email. Did you get it?
% Yes! Okay, I have to change train now, but will submit in the next train! ok

\begin{thm}
\label{thm:gw_equivariant_CY}
    Let $k>0$.% $X=\mathcal{O}(a_1)\oplus\mathcal{O}(a_2)\rightarrow \mathbb{P}^1$ be Calabi--Yau and let $k:=(a_1-a_2)/2 > 0$.
    \begin{enumerate}
        \item If $X_k$ is equivariantly Calabi--Yau, i.e. $t_3=-t_1-t_2$, then
        \begin{equation}\label{eq:GW_thm_1}
            GW_{0,0}^{X_k,d\mathbf{0}_{X_k}} = \frac{(-1)^{d(k+1)-1}}       {d^3k^2}\binom{k^2 d}{d}.
        \end{equation}
            \label{item:GW_thm_1}
        \item  
        % If $X_k$ belongs to the \emph{"other"} case, 
        If $t_3 = t_2- kt_1$, then
             \[
            GW_{0,0}^{X_k,d\mathbf{0}_{X_k}} = \frac{1}{d^3}.
            \]
            \label{item:GW_thm_2}

        \item If $k=1$ and $t_3=-t_1+yt_2$ for any non-zero integer $y$, then
            \[
            GW_{0,0}^{X_k,d\mathbf{0}_{X_k}} = \frac{y}{d^3}.
            \]
            \label{item:GW_thm_3}
    \end{enumerate}
\end{thm}
\begin{proof}
    Our strategy is to understand how the formula in Theorem~\ref{thm:LS17} behave when applied to our setting.  The GKM graph of $X_k$ is shown in Figure~\ref{fig:localP1}. 
    Let $e$ be the edge of $X_k$. By Equation~\eqref{eq:h}, for any $m\ge 1$ we have
    \begin{equation}\label{eq:h_in_5.3}
        h(e,m) = \frac{(-1)^m m^{2m}}{(m!)^2 t_1^{2m}}\cdot b(t_1/m, t_2, m(k-1)) \cdot b(t_1/m, t_3, m(-k-1)),
    \end{equation}
    where
    \begin{align*}
        b(t_1/m, t_2, m(k-1)) &= \prod_{j=0}^{m(k-1)}\left(t_2 - \frac{j}{m}t_1\right)^{-1} \\
        &= \prod_{j=m}^{mk}\left(t_2 - \left(\frac{j}{m}-1\right) t_1\right)^{-1},
    \end{align*}
    and
    \begin{align}\label{eqn:b_a2}
        b(t_1/m, t_3, m(-k-1))
        &=
        \prod_{j=1}^{m(k+1)-1}\left(t_3 + \frac{j}{m}t_1\right).
    \end{align} 
    In case \eqref{item:GW_thm_1}, that is $t_3=-t_1-t_2$, an easy refactoring shows that the right hand side of Equation~\eqref{eqn:b_a2} is equal to
    $$(-1)^{m(k+1)-1}\prod_{j=1}^{mk}\left(t_2 - \left(\frac{j}{m}-1\right)t_1\right)
    \prod_{j=1}^{m-1}\left(t_2 - \left(k+\frac{j}{m}-1\right)t_1\right).$$
    From Equation~\eqref{eq:h_in_5.3} it is not difficult to deduce that 
    \begin{equation*}
        h(e,m) = \frac{(-1)^{mk-1} m^{2m}}{(m!)^2 t_1^{2m}}\prod_{j=1}^{m-1}\left(t_2 + \frac{j}{m}t_1\right)\left(t_2 + \left(\frac{j}{m}-k\right)t_1\right).
    \end{equation*}
    We can compute $h(e, m)$ also in the remaining two cases, arguing in the same way. The results for cases \eqref{item:GW_thm_2} and \eqref{item:GW_thm_3} are, respectively,
    \begin{align*}
        h(e,m) &= \frac{(-1)^{m} m^{2m}}{(m!)^2 t_1^{2m}}\prod_{j=1}^{m-1}\left(t_2 + \frac{j}{m}t_1\right)\left(t_2 + \left(\frac{j}{m}-k\right)t_1\right)\\
        h(e,m) &= \frac{(-1)^{m} m^{2m}}{(m!)^2 t_1^{2m}} y \prod_{j=1}^{m-1} 
        \left(yt_2 + \frac{j}{m}t_1\right)
        \left(yt_2 + \left(\frac{j}{m}-1\right)t_1\right).
        % h(e,m) &= \frac{(-1)^{m} m^{2m}}{(m!)^2 t_1^{2m}} y \prod_{j=1, j\neq m}^{2m-1}        \left(yt_2 + \left(\frac{j}{m}-1\right)t_1\right).
    \end{align*}
    From the formulas above, we proved that, in the three cases, $h(e,m)$ is the quotient of two polynomials where the denominator is a power of $t_1$. Moreover, we have
    $$\alpha_{(\eta,v)}^{-1}=-\alpha_{(\eta,v_0)}^{-1}=\frac{m}{t_1}\,\,\,\,\,\forall \eta=\{v_0,v\} \in E(\Gamma),\overrightarrow{f}(v)=[1:0].$$
    In particular, since $GW_{0,0}^{X_k,d\mathbf{0}_{X_k}}$ is the quotient of two polynomials of the same degree (see Section~\ref{sec:GW}), $GW_{0,0}^{X_k,d\mathbf{0}_{X_k}}\in\mathbb{Q}[t_2/t_1]$, that is, it 
    is a polynomial in $t_2/t_1$. 

    By Lemma~\ref{lem:polarity}, $GW_{0,0}^{X_k,d\mathbf{0}_{X_k}}$ is constant. We can compute it by setting $t_2=0$, $t_1=1$. Let $v$ be a vertex of a decorated tree. From
    \begin{align*}
        \alpha([1:0]) &= t_1t_2t_3 \\
        \alpha([0:1]) &= -t_1(t_2-(k-1)t_1)(t_3+(k+1)t_1), 
    \end{align*}
    % \begin{equation*}  
    %     \alpha(v) = \begin{cases}
    %         t_1t_2t_3 & \text{if }\overrightarrow{f}(v)= [1:0]\\
    %         -t_1(t_2-(k-1)t_1)(t_3+(k+1)t_1) & \text{if }\overrightarrow{f}(v)= [0:1].
    %     \end{cases}
    % \end{equation*}
    it follows that $\alpha(\overrightarrow{f}(v)) = 0$ if $v$ is mapped to $[1:0]$. The fact that $\alpha(\overrightarrow{f}(v))$ appears with exponent $\mathrm{val}(v)-1$ in the $GW$ formula, implies that the contribution from a decorated tree is zero when it has a vertex mapped to $[1:0]$ with valency strictly greater than one. Hence the contributions to $GW_{0,0}^{X_k,d\mathbf{0}_{X_k}}$ come only from star-like trees having the central vertex mapped to $[0:1]$ and all other vertices mapped to $[1:0]$. 
    
    The set of all such decorated trees is in bijective correspondence with the set of partitions of $d$. We denote the partitions of $d$ by $d$-vectors $(a_1,\ldots, a_d)$ of non-negative integers such that $\sum_{j=1}^d ja_j=d$. The bijection is realized in the following way: the partition $(a_1,\ldots, a_d)$ corresponds to the tree having exactly $a_j$ edges of degree $j$, for $1\le j\le d$. This decorated graph has $a_1!a_2!\cdots a_d!$ automorphisms.

    The formula in Theorem~\ref{thm:LS17} can be rewritten in terms of partitions of $d$. The contribution of each star-like graph with central vertex $v_0$ is 
    \begin{equation}
        \label{eq:contribution}
        \frac{1}{a_1!\cdots a_d!}
    \frac{\alpha([0:1])^{\mathrm{val}(v_0)}}{\alpha([0:1])}
    \frac{(-d)^{\mathrm{val}(v_0)}}{-d^{3}}
    \prod_{j=1}^d \left(-\frac{h(e,j)}{j} \right)^{a_j}.
    \end{equation}    
    If $k=1$, then $\alpha([0:1])=0$, thus Equation~\eqref{eq:contribution} is non-zero only if $\mathrm{val}(v_0)=1$. That is, the contribution comes only from the decorated graph with just one edge. This contribution is easy to compute in all three cases. So let us suppose $k\ge 2$.
    
    From $\mathrm{val}(v_0)=\sum_{j=1}^da_j$, in the first case we deduce that
    $$GW_{0,0}^{X_k,d\mathbf{0}_{X_k}} =
    \frac{-1}{d^3k(k-1)}\sum_{(a_1,\ldots,a_d)}\frac{1}{a_1!\cdots a_d!}\prod_{j=1}^d \left (\frac{(-1)^{j(k-1)}(k-1)d}{j}\binom{kj}{j}\right )^{a_j}.
    $$
    After canceling common signs and factors, Equation~\eqref{eq:GW_thm_1} holds true if and only if
    \begin{equation}\label{eq:final}
        \sum_{(a_1,\ldots,a_d)}\frac{1}{a_1!a_2!\cdots a_d!}\prod_{j=1}^d \left (\frac{(k-1)d}{j}\binom{kj}{j}\right )^{a_j}
    =
    \frac{k-1}{k}\binom{k^2d}{d}.
    \end{equation}
    This equation follows from Lemma~\ref{Lemma:Eisenkoelbl}  below by taking $t=(k-1)d$.

    In case \eqref{item:GW_thm_2}, we have
    $$GW_{0,0}^{X_k,d\mathbf{0}_{X_k}} =
    \frac{1}{d^3(1-k)}\sum_{(a_1,\ldots,a_d)}\frac{1}{a_1!a_2!\cdots a_d!}\prod_{j=1}^d \left (\frac{(1-k)d}{kj}\binom{kj}{j}\right )^{a_j}.
    $$
    By applying Lemma~\ref{Lemma:Eisenkoelbl} with $t=(1-k)d/k$, we see that this expression is $1/d^3$. This concludes the proof.
\end{proof}

\begin{cor}
    Let $k>0$. 
    % $X=\mathcal{O}(a_1)\oplus\mathcal{O}(a_2)\rightarrow \mathbb{P}^1$ be Calabi--Yau and assume that $k:=|a_1-a_2|/2>0$.
    The degree zero rational function $GW_{0,0}^{X_k,d\mathbf{0}_{X_k}}$ in $t_1,t_2,t_3$ simplifies to a rational number for all $d$ if and only if $GW_{0,0}^{X_k,\mathbf{0}_{X_k}}$ does.
\end{cor}
Theresia Eisenk{\"o}lbl conjectured the following lemma in the course of a discussion on Equation~\eqref{eq:final}.
\begin{lem}\label{Lemma:Eisenkoelbl}
        Let $d$ and $k$ be positive integers. Let us denote by $(a_1,\ldots, a_d)$ any $d$-vector of non-negative integers such that $\sum_{j=1}^d ja_j=d$. Then the following formula holds in $\mathbb{Q}[t]$
        \begin{equation}
            \label{eq:Eisenk}
    \sum_{(a_1,\ldots, a_d)}\frac{1}{a_1!a_2!\cdots a_d!}\prod_{j=1}^d \left (\frac{t}{j}\binom{kj}{j}\right )^{a_j}
    =
    \frac{t}{d+t}\binom{k(d+t)}{d}.
        \end{equation}
\end{lem}

\begin{proof}
    Both sides of the equation are clearly in $\mathbb{Q}[t]$. Let us suppose that $t$ is a positive integer. 
    Let $C_{k,j}$ be the Fuss--Catalan number
    $$C_{k,j} = \frac{1}{(k-1)j+1}\binom{kj}{j},$$
    and let $G_k(x) = 1+\sum_{j\ge1} C_{k,j}x^j$ be the generating function.
    Let $F(x)$ be the power series
    $$F(x):=1+\sum_{d\ge1} \sum_{ (a_1,\ldots, a_d)}\frac{1}{a_1!a_2!\cdots a_d!}\prod_{j=1}^d \left (\frac{t}{j}\binom{kj}{j}\right )^{a_j} x^d.$$
    Using $\exp(ab)=\exp(b)^a$ and the product formula (see \cite[Equation~(38), 1.2.9]{MR3077152}), 
    $$
    F(x) = \exp\left( \sum_{j\ge 1} \frac{t}{j}\binom{kj}{j} x^j \right)
    = \exp\left( \sum_{j\ge 1} \frac{1}{j}\binom{kj}{j} x^j \right)^t.
    $$
    By \cite[Theorem~8]{Jansen_Kolesnikov_2024},
    $$
    G_k(x) = \exp\left( \sum_{j\ge 1} \frac{1}{kj}\binom{kj}{j} x^j \right)
    = \exp\left( \sum_{j\ge 1} \frac{1}{j}\binom{kj}{j} x^j \right)^{1/k}.
    $$
    Hence, $F(x) = G_k(x)^{kt}$. Let us define
    \begin{equation}
        A_d(n, k) := \frac{n}{-kd+n}\binom{-kd+n}{d}.
    \end{equation}
    We need to prove that for each $n\ge 1$, 
    \begin{equation}\label{eq:Gk(x)n}
        G_k(x)^n = \sum_{d\ge0} A_d(n, -k) x^d.
    \end{equation}
    The case $n=1$ is easy as the following equality is immediate
    $$
    \frac{n}{kd+n}\binom{kd+n}{d} = \frac{n}{(k-1)d+n}\binom{kd+n-1}{d}.
    $$
    Let us suppose Equation~\eqref{eq:Gk(x)n} holds true for $n$, hence $G_k(x)^{n+1}$ is
    \begin{align*}
        G_k(x)G_k(x)^{n} &= \sum_{d\ge0}\left(\sum_{j=0}^d  A_j(1, -k) A_{d-j}(n, -k)\right) x^d \\
        &=\sum_{d\ge0} A_d(n+1, -k) x^d.
    \end{align*}
    In the last equation, we used \cite[Example~4, 1.2.6]{MR3077152}. Finally, by taking $n=kt$ we prove the lemma for any positive integer $t$. Thus Equation~\ref{eq:Eisenk} holds in $\mathbb{Q}[t]$.
\end{proof}
\subsubsection{BPS-invariants}\label{sec:BPS}
We now apply equation \eqref{eqn:GV_genus_zero} to turn the Gromov--Witten invariants of Theorem~\ref{thm:gw_equivariant_CY} into genus zero BPS-invariants.
In the second case of \eqref{eqn:polynomial_cases}, we simply get $n_{0,\mathbf{0}_{X_k}}(X_k)=1$ and $n_{0,d\mathbf{0}_{X_k}}(X_k)=0$ for $d>1$.
In the third case of \eqref{eqn:polynomial_cases}, we get $n_{0,\mathbf{0}_{X_k}}(X_k)=y$ and $n_{0,d\mathbf{0}_{X_k}}(X_k)=0$ for $d>1$.
On the other hand, the equivariantly Calabi--Yau case is much more interesting.
Applying the M{\"o}bius inversion formula, we get
\begin{align}
    \label{eqn:BPS_equivariantly_CY}
    n_{0,d\mathbf{0}_{X_k}}(X_k) =
    \frac{1}{d^3 k^2}\sum_{e\mid d} \mu(d/e)(-1)^{(k+1)e+1}\binom{k^2 e}{e},
\end{align}
where $\mu$ is the M{\"o}bius function.
Table~\ref{tab:BPS_numbers} shows the resulting BPS-invariants.
% \daniel{(3RDV) These are integers by (?) \cite[3.3 Local curves]{MR2264664}}.
% \daniel{Study \cite{vangarrel2025} and \cite{MR4157555}.}
% \daniel{Giosue, are you absolutely sure that this is really proved there? I can only find an explicit statement for $\mathcal{O}(-1)\oplus\mathcal{O}(-1)$...}
% \giosue{I am not so sure now. My interpretation of the second statement --- The proof of Conjectures 2 and 3
% for local surfaces given in Section 4 below is valid for the local $\mathbb{P}^1$
% case --- was that every local CY 3 over $\mathbb{P}^1$ (not just $\mathcal{O}(-1)\oplus \mathcal{O}(-1)$) satisfies Conjectures 2 and 3. Those two conjectures imply, using numerical manipulation, that DT and BPS coincides\footnote{I did not really check this last statment, I was just guessing based on plausibility}, thus BPS are integers. My argument seems weak.} 
% \daniel{I will ask Yannik! (NP).}

\begin{table}[htb]
    \centering
    \begin{tabular}{|c||c|c|c|c|c|c|c|}
    \hline
        $k\downarrow\,\,\,d\rightarrow$ & $1$ & $2$ & $3$ & $4$ & $5$ & $6$ & $7$ \\\hline\hline
        $0$ & $1$ & $0$ & $0$ & $0$ & $0$ & $0$ & $0$ \\\hline
        $1$ & $-1$ & $0$ & $0$ & $0$ & $0$ & $0$ & $0$ \\\hline
    $2$ & $1$ & $-1$ & $2$ & $-7$ & $31$ & $-156$ & $863$ \\\hline
    $3$ & $-1$ & $-2$ & $-12$ & $-102$ & $-1086$ & $-13284$ &   $-179226$ \\\hline
    $4$ & $1$ & $-4$ & $40$ & $-620$ & $12020$ & $-268248$ &    $6601292$ \\\hline
    $5$ & $-1$ & $-6$ & $-100$ & $-2450$ & $-75050$ & $-2647580$ &  $-102998030$ \\\hline
    \end{tabular}
    \vspace{2mm}
    \caption{BPS-invariants $n_{0,d\mathbf{0}_{X_k}}(X_k)$ in genus zero for equivariantly Calabi--Yau $X_k$.
    }
    \label{tab:BPS_numbers}
\end{table}

Note that Equation~\eqref{eqn:BPS_equivariantly_CY} can be expressed also as
\begin{equation}\label{eqn:BPS_DT}
    n_{0,d\mathbf{0}_{X_k}}(X_k) =\frac{1}{d(k+1)}D(d, d, k+1),
\end{equation}
where $D(d, d, m)$ is the diagonal Donaldson--Thomas invariant of degree $d$ of the Kronecker quiver $K_m$ \cite[Theorem~5.2]{MR2801406_Reineke_2011} (see also \cite[Corollary~3.3.2 \& Table~B.1]{mainiero2015beyond}).
%\begin{equation*}
%    n_{0,d\mathbf{0}_{X_k}}(X_k) =\frac{1}{d(k+1)}\Omega_{k+1}%(d\gamma_c),
%\end{equation*}
%where $\Omega_{m}(d\gamma_c)$ is the diagonal Donaldson--Thomas %invariant of degree $d$ of the Kronecker quiver $K_m$ (see %\cite[Corollary~3.3.2]{mainiero2015beyond}).
Expressions very similar to \eqref{eqn:BPS_equivariantly_CY} also appear in literature concerning the \emph{GW--Kronecker correspondence} (see \cite{MR4157555, MR3004575_Reineke_Weist_2013, MR2662867_Gross_Pandharipande_2010, MR2667135_Gross_Pandharipande_Siebert_2010, Kontsevich_Soibelman_2009}), which considers Gromov--Witten invariants of certain toric surfaces.

We thank Pierrick Bousseau for pointing out that Equation~\eqref{eqn:BPS_DT} was observed independently by \cite[p.36]{vangarrel2025}, as explained in the Introduction.

\subsubsection{An alternative approach for proving rationality}
A key step in the above proofs was to show that $GW_{0,0}^{X_k,d\mathbf{0}_{X_k}}\in\mathbb{Q}$.
This allowed us to conveniently set $t_2=0$ and reduce the calculation to a specific class of trees; a standard strategy in calculations like this (see \cite{Manin_1995}).

Note that this step is non-trivial and only holds in the specific assumptions of Theorem~\ref{thm:gw_equivariant_CY}, as demonstrated by Lemma~\ref{lem:polynomiality_condition}.
Another strategy for completing this step is as follows. Try to find a proper algebraic or compact Hamiltonian GKM space $Z$ with $\dim_{\mathbb{C}}Z=3$ whose GKM graph $G$ contains an edge $e$ such that
\begin{align}
\label{eqn:isolated_edge}
\mathbb{Z}_{>0}\cdot [C_e]\cap\text{span}_{\mathbb{N}_{\ge0}}\{[C_{e'}]:e'\in E(G)\setminus\{e\}\}
=
\emptyset,
\end{align}
and such that $G$ looks like $X_k$ (see Figure~\ref{fig:localP1}) locally at $e$ with $t_1,t_2,t_3,k$ as in the hypothesis of the Theorem.
Hence, any Gromov--Witten invariant calculated in that local model will automatically be in $\mathbb{Q}$.
Indeed, the invariants of the local model coincide with the invariants on $Z$ in classes $d[C_e]$.
But $Z$ is compact so its Gromov--Witten invariants are polynomials in the equivariant parameters.
As $GW_{0,0}^{X_k,d\mathbf{0}_{X_k}}$ has degree zero, it must lie in $\mathbb{Q}$.
One particularly nice case where \eqref{eqn:isolated_edge} holds is when $G$ is almost positive in the following sense.

\begin{defn}[Almost positive]\label{def:almost_positive}
    An abstract GKM graph $G$ is \emph{almost positive} if there exists $e_0\in E(G)$ such that $\mathcal{C}_1(e)>0$ holds for every $e\in E(G)\setminus\{e_0\}$ and $\mathcal{C}_1(e_0)=0$.
    We call $e_0$ the \emph{exceptional edge} of $G$.
    We call a GKM space \emph{almost positive} if its GKM graph is.
\end{defn}

For the equivariantly Calabi--Yau case with $k=2$, the local model $X_k$ appears as exceptional edge in the almost positive GKM graph of the twisted flag manifold, proving rationality of $GW_{0,0}^{X_k,d\mathbf{0}_{X_k}}$ in this case (see Section~\ref{sec:twisted_flag} and Figure~\ref{fig:F3_twisted}).

For the equivariant Calabi--Yau case with $k=1$, the GKM graph in Figure~\ref{fig:example_3d_k1} contains two edges whose neighbourhood is $X_k$.
Both of them satisfy \eqref{eqn:isolated_edge} and there is a compact Hamiltonian GKM space with this GKM graph by \cite[Theorem~5.1 \& \S6.2]{GKZ20}, so $GW_{0,0}^{X_k,d\mathbf{0}_{X_k}}$ is rational in this case as well.

\begin{question}
    For which $k\ge 3$ does there exist a Hamiltonian GKM space $X$ whose GKM graph $G$ has an edge $e$ satisfying \eqref{eqn:isolated_edge} such that $G$ looks locally like $X_k$ at $e$?
\end{question}
% \textbf{Question:} For which $k\ge 3$ does there exist a Hamiltonian GKM space $X$ whose GKM graph $G$ has an edge $e$ satisfying \eqref{eqn:isolated_edge} such that $G$ looks locally like $X_k$ at $e$?

\begin{figure}
    \centering
    \includegraphics[width=0.3\linewidth]{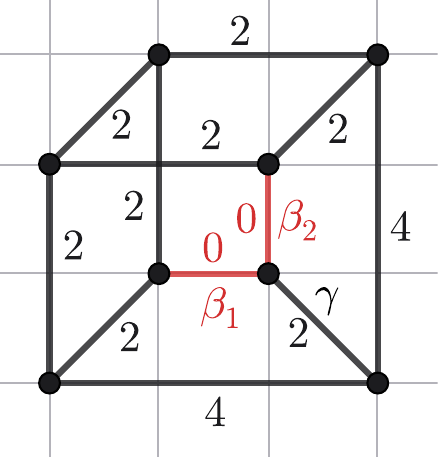}
    \caption{GKM graph of a compact Hamiltonian GKM space $X$ (see \cite[after~Corollary~2.13]{GKZ20}), illustrated following Example~\ref{ex:GKM_illustration}~(2).
    Each edge $e$ is labeled with $\mathcal{C}_1(e)$.
    By Proposition~\ref{prop:curve_classes}, $H_2(X)$ is the free $\mathbb{Z}$-module generated by $\beta_1,\beta_2$, and $\gamma$.}
    \label{fig:example_3d_k1}
\end{figure}

\subsection{The twisted flag variety}
\label{sec:twisted_flag}

There are (at least) three different constructions of a Hamiltonian GKM space that give rise to the GKM graph depicted in Figure~\ref{fig:F3_twisted}.
We refer to \cite[\S4]{MR4088417_GKZ2020_Eschenburg_Woodward_Tolman} for an overview.

\begin{figure}[htb]
    \centering
    \includegraphics[width=0.35\linewidth]{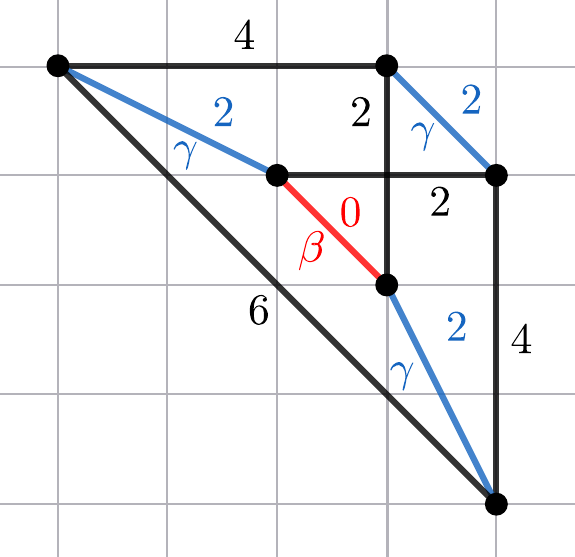}
    \caption{GKM graph of the twisted flag variety illustrated as in Example~\ref{ex:GKM_illustration}~(2). We have $H_2(X)=\langle\beta, \gamma \rangle$. Each edge $e$ is decorated with $\mathcal{C}_1(e)$.
    }
    \label{fig:F3_twisted}
\end{figure}

One possible construction is as biquotient
\[
X:= SU(3)//T,\hspace{1cm}
(s,t)\cdot A = \begin{pmatrix}
    s^2 t^2 & & \\
    & 1 & \\
    & & 1
\end{pmatrix}
A
\begin{pmatrix}
    \overline{s} &&\\
    &\overline{t}&\\
    &&\overline{st}
\end{pmatrix}
\]
as described in \cite{MR4142487_GKZ2020_Eschenburg} and named \emph{Eschenburg flag} after
\cite{MR758252_Eschenburg_habil,MR1245552_Eschenburg_paper}.
Another construction by Tolman \cite{MR1608575_Tolman_1998} works by gluing together two pieces of Hamiltonian $T$-spaces, while a similar construction of Woodward \cite{MR1608579_Woodward} uses symplectic surgery on the full flag manifold $SU(3)/(S^1)^3$.
By \cite[\S5.1]{GKZ22} these spaces are all equivariantly homeomorphic to each other and (non-equivariantly) diffeomorphic.
Abstractly, the GKM graph appears as a 3D GKM fibration over the GKM graph of $\mathbb{CP}^2$ \cite[Example~4.8]{GKZ20}.

\begin{defn}
\label{def:twisted_flag}
    A \emph{twisted flag manifold} is a Hamiltonian GKM space whose GKM graph agrees with that of Figure 
\ref{fig:F3_twisted}.
\end{defn}

These spaces are interesting to study for the following two reasons:
\begin{enumerate}
    \item They are Hamiltonian GKM spaces that do not admit a $T$-compatible K{\"a}hler structure \cite[\S7]{GKZ20}.
    \item Their GKM graph is almost positive (see Definition~\ref{def:almost_positive}).
\end{enumerate}
The first condition implies that these examples do not arise as algebraic GKM spaces.
The second condition implies via Theorem~\ref{thm:LS17} that the calculation of Gromov--Witten invariants in any curve class $\beta$ with $\int_\beta c_1(T_X)=0$ reduces to a calculation on the local model
\[
\mathcal{O}_{\mathbb{CP}^1}(1)\oplus\mathcal{O}_{\mathbb{CP}^1}(-3)\longrightarrow {\mathbb{CP}^1}.
\]
The GKM graph is precisely that of $X_k$ in Figure~\ref{fig:localP1} with $k=2$.
The unique compact edge corresponds to $\beta$ in Figure~\ref{fig:F3_twisted}.
In fact, the weights of the edge labelled $\beta$ sum to zero at both vertices, so we are in the equivariantly Calabi--Yau case of the local model (see Section~\ref{sec:our_local_setting}).

The following theorem holds in particular for the twisted flag manifolds, where $k=2$ and the exceptional edge is the one labelled $\beta$ in Figure \ref{fig:F3_twisted}.
The cases are aligned with those of \eqref{eqn:polynomial_cases} and Theorem~\ref{thm:gw_equivariant_CY}.

\begin{thm}\label{thm:twisted_Qprod}
    Let $X$ be an almost positive Hamiltonian or algebraic GKM space with $\dim_{\mathbb{R}}X=6$ and let $e$ be the exceptional edge of its GKM graph (see Definition~\ref{def:almost_positive}).
    Let $k,t_1,t_2,t_3$ be the parameters for $N_{C_e/X}$ so that $N_{C_e/X}\cong X_k$ as in Figure~\ref{fig:localP1}.
    For any $T$-stable curve $C\subset X$, let $\mathrm{PD}(C)\in H_T^4(X)$ be its equivariant Poincar{\'e} dual.
    Let $x_1,x_2\in H_T^*(X)$.
    Then
    \begin{align}
    \label{eqn:almost_pos_prod}
            x_1\ast_T x_2 &= x_1\smile x_2 +\mathrm{PD}(C_e)\left(\int^T_{C_e}x_1\right) \left(\int^T_{C_e}x_2\right) \underbrace{\sum_{d\ge 1}GW_{0,0}^{X_k,d[C_e]} d^3q^{d[C_e]}}_{(\dagger)}
            \\
            & + \left(q^\alpha\text{-terms with }\int_\alpha c_1(T_X) > 0\right).
            \nonumber
    \end{align}
    Moreover, $(\dagger)$ is given by
    \[
    (\dagger)=
\begin{cases}
\frac{q^{[C_e]}}{1-q^{[C_e]}}
& \text{if }k=0\,\,\text{ or }\,\,t_3=t_2-kt_1\text{ with }k\ge 1
\\
y\frac{q^{[C_e]}}{1-q^{[C_e]}}
& \text{if }k=1 \text{ and }t_3=-t_1+yt_2
\\
\sum_{d\ge 0}\frac{(-1)^{d(k+1)+1}}{k^2}\binom{k^2d}{d}q^{d[C_e]}
& \text{if }k\ge 1\text{ and } t_3=-t_1-t_2
\\
\end{cases}
\]
and at least one of those cases applies.
\end{thm}

\begin{proof}
    The $q^0$ term of $x_1\ast_T x_2$ is always $x_1\smile x_2$.
    Since $e$ is the only edge with $\mathcal{C}_1(e)=0$ and we have $\mathcal{C}_1(e')>0$ for all other edges $e'$, all relevant Gromov--Witten invariants can be computed over the local model $X_k\cong N_{C_e/X}\rightarrow C_e$ in Figure~\ref{fig:localP1}, involving only positive multiples of $[C_e]$, which corresponds to the class $\mathbf{0}_{X_k}\in H_2(X_k)$ of the zero section.
    Let $i\colon C_e\hookrightarrow X$ be the embedding of $C_e$ into $X$.

    By the defining properties
    \[
    \int_X^T (x_1\ast_T x_2)\smile z= \sum_{\beta\in H_2(X)} GW_{0,3}^{X,\beta}(x_1,x_2,z) q^\beta
    \]
    of $\ast_T$ and
    \[
    \int^T_X \mathrm{PD}(C_e)\smile z = \int^T_{C_e} i^*z
    \]
    of $\mathrm{PD}(C_e)$, it suffices to prove the following equation.
    \[
        GW_{0,3}^{X,d[C_e]}(x_1,x_2,z) = d^3\left(\int^T_{C_e}x_1\right) \left(\int^T_{C_e}x_2\right)\left(\int^T_{C_e}z\right)GW_{0,0}^{X_k,d\mathbf{0}_{X_k}}.
    \]
    Pick $x_1',x_2',z'\in H_T^*\cdot H_T^{\le 2}(X)$ such that $i^*x_1'=i^*x_1$, $i^*x_2'=i^*x_2$, and $i^*z'=i^*z$.
    This is possible because $H_T^*(X)\rightarrow H_T^*(C_e)$ is surjective.
    Indeed, it is surjective in degree 0 by $H_T^*$-linearity and it is surjective in degree 2 because some $w\in H_T^*(X)$ has non-zero pairing with $[C_e]$ (cf. Lemma~\ref{lem:torsion_free}).
    Since all calculations take place over the local model $X_k$, we have
    \[
     GW_{0,3}^{X,d[C_e]}(x,y,z) =  GW_{0,3}^{X,d[C_e]}(x',y',z').
    \]
    By the \emph{fundamental class axiom} and the \emph{divisor axiom} on $X$, we obtain
    \[
        GW_{0,3}^{X,d[C_e]}(x',y',z')
        = \left(\int^T_{d[C_e]}x'\right) \left(\int^T_{d[C_e]}y'\right)\left(\int^T_{d[C_e]}z'\right)GW_{0,0}^{X,d[C_e]}.
    \]
    The proof of \eqref{eqn:almost_pos_prod} is completed by noting that $\int^T_{C_e}x_1'=\int^T_{C_e}x_1$ (and similarly for $x_2,z$), that $\int_{d[C_e]}(\dots)=d\int_{C_e}(\dots)$, and that $GW_{0,0}^{X,d[C_e]} = GW_{0,0}^{X_k,d\mathbf{0}_{X_k}}$ by locality of the calculation.

    Finally, the expression for $(\dagger)$ follows from Theorem~\ref{thm:gw_equivariant_CY} and the sentence before it regarding the well-known case $k=0$.
    Note that no other case can occur by Lemma~\ref{lem:polynomiality_condition}, since $GW_{0,0}^{X_k,\mathbf{0}_{X_k}}$ must be a polynomial in $t_1,t_2,t_3$ (cf. Proposition~\ref{prop:realizability_almost_positive} below).
\end{proof}

\begin{rem}
    The point of introducing $x',y',z'$ in the above proof is so that we can apply the fundamental class and divisor axioms on the compact space $X$ rather than on the non-compact $X_k$.
    %\daniel{Giosue, do you know if the Kontsevich--Manin axioms also hold for non-compact target spaces? My understanding is that for non-compact spaces, we use localization to \emph{define} GW invariants. But it is not obvious from the localization formula why the axioms should hold. Do you know a reference for this?}
    %\giosue{Great question. Main classical references on GW invariants suppose X projective. I asked chatgpt and it says that if you take the Equivariant formula as definition of GW invariants, then KM axioms holds. As a reference, it reports Liu Shesmani but I couldn't find that discussion there. We could add as it seems missing in the literature.}
    
    % \daniel{(NP) Is there a source for Kontsevich--Manin axioms for non-proper / quasi-projective spaces? D: Email some more people!}
\end{rem}

As an example application of {\pkg} (see Section~\ref{sec:tool}), we calculate a few more quantum products on twisted flag manifolds.

\begin{thm}
\label{thm:twisted_numbers}
    Let $X$ be a twisted flag manifold and $p:=\mathrm{PD}([\mathrm{pt}])\in H^*(X)$. Then
    % \begin{multline}
    %     p\ast p = q^{3\gamma}\left(
    %         q^{2\beta} - 5q^{3\beta} + 35 q^{4\beta} - 273 q^{5\beta} + 2244q^{6\beta}-19019q^{7\beta}+\dots
    %     \right) \\
    %     + \left(q^\alpha\text{-terms with }\int_\alpha c_1(T_X) < 6 \right).
    % \end{multline}
    \begin{multline}\label{eqn:twisted_F3_p_times_p}
        p\ast p = q^{3\gamma}\left(
            q^{2\beta} - 5q^{3\beta} + 35 q^{4\beta} - 273 q^{5\beta} + 2244q^{6\beta}-19019q^{7\beta}+\dots
        \right) \\
        + \left(q^{a\gamma + n\beta}\text{-terms with } 0\le a\le 2 \right).
    \end{multline}
\end{thm}
\begin{proof}
    Since $\ast$ preserves the grading and $p\in H^6(X)$, every summand of $p\ast p$ is of the form $A q^\delta$ for some $A\in H^{|A|}(X;\mathbb{Q})$, $\delta\in H_2^\text{eff}(X)$, where $|A|+2\int_\delta c_1(T_X) = 12$.
    By Figure~\ref{fig:F3_twisted}, every $\delta$ yielding a non-zero Gromov--Witten invariant is of the form
    \[
    \delta= b\beta + c\gamma
    \]
    for some $b,c\in\mathbb{Z}_{\ge 0}$.
    Since $\int_\beta c_1(T_X)=0$ and $\int_\gamma c_1(T_X)=2$, we have $\int_\delta c_1(T_X)=2b$ and the degree condition translates into $|A|+4b=12$.
    Since $|A|$ is non-negative, the only terms in $p\ast p$ which are not of the form $Aq^\delta$ with $\int_\delta c_1(T_X)<6$ are therefore those with $b=3$ and $|A|=0$, i.e. $A\in\mathbb{Q}$.
    Hence, in \eqref{eqn:twisted_F3_p_times_p} we only need to consider $q^\delta$ terms with $\delta = 3\gamma + n\beta$ for $n\in\mathbb{Z}_{\ge 0}$.

    To calculate the rational coefficients of these terms for the first few $n$, we use \pkg.
    The precise code leading to the numbers in \eqref{eqn:twisted_F3_p_times_p} is available in the online documentation (see Section~\ref{sec:tool}).
\end{proof}
\begin{rem}
    We have not identified a closed form for the sequence
    \[
    (0, 0, 1, -5, 35, -273, 2244, -19019, \dots)
    \]
    occurring as coefficients in \eqref{eqn:twisted_F3_p_times_p}. More terms may be found in \cite[A388298]{oeis}.
    However, integrality of the terms is explained as follows.
    First, $X$ is \emph{semipositive} in the sense of \cite[Definition~6.4.1]{McDuff_Salamon_2012} because $\dim_{\mathbb{R}}(X)= 6$.
    Thus, the construction of Gromov--Witten invariants for semipositive symplectic manifolds applies, which yields integer invariants for the curve classes at hand \cite[Theorem~7.1.1,~\S7.5]{McDuff_Salamon_2012}.
    Second, this construction agrees with the definition via global Kuranishi charts by \cite[Theorem~6.4]{Hirschi_2023}.
    Finally, the global Kuranishi chart invariants are precisely what the localization formula computes by Theorem~\ref{thm:Hirschi}.
    Hence, the sequence consists of integers.
\end{rem}

\subsection{Realizability of GKM graphs}
\label{sec:realizability}

So far, we have used GKM theory to calculate certain Gromov--Witten invariants.
Now we reverse this logic and demonstrate how Gromov--Witten theory can be applied to the \emph{realizability problem} of GKM theory.
Given an abstract GKM graph, this problem asks whether there exists a GKM space whose GKM graph is the given one.
The question has many interesting refinements by asking for specific $T$-compatible geometric structures that the GKM space should have, for example a $T$-compatible K{\"a}hler structure \cite[\S2]{GKZ20}.
The twisted flag varieties of Section~\ref{sec:twisted_flag} provide examples of Hamiltonian GKM spaces that do not admit a compatible K{\"a}hler structure.

By Theorem~\ref{cor:GKM_determines_GW} the Gromov--Witten invariants of Hamiltonian or algebraic GKM spaces are determined purely by their GKM graph.
Theorem~\ref{thm:LS17}, Proposition~\ref{prop:curve_classes} and Theorem~\ref{thm:con_indep} give an algorithm to compute them.
However, from Theorem~\ref{thm:LS17} it is not at all obvious that the Gromov--Witten invariants calculated using this formula satisfy the Kontsevich--Manin axioms. It is not even clear a priori why they should give polynomials in the equivariant parameters rather than rational functions.
In fact, it is not true that these properties hold for any abstract GKM graph as demonstrated by the examples below.
We can therefore derive new necessary criteria for Hamiltonian or algebraic realizability of GKM graphs by determining when their GW invariants are well-behaved.

We only demonstrate this idea with a few examples here.
We intend to develop this further in future work.

\subsubsection{Example 1: a non-Hamiltonian 2D GKM graph}
Let $G$ be the 2D GKM graph given by a cycle of length 8 with weights
\[
t_1, t_2, -t_1, -t_2,t_1, t_2, -t_1, -t_2
\]
in cyclic order, as depicted in Figure~\ref{fig:GKM_non_Hamiltonian}.
\begin{figure}
    \centering
    \includegraphics[width=0.27\linewidth]{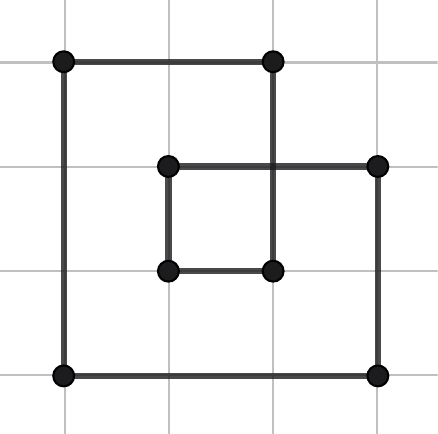}
    \caption{An abstract GKM graph not realizable by any Hamiltonian GKM space, illustrated using the method of Example~\ref{ex:GKM_illustration}~(2).
    The same GKM graph is discussed in \cite[Example~2.44]{GKZ22}.}
    \label{fig:GKM_non_Hamiltonian}
\end{figure}
Let $\beta$ be the curve class corresponding to the first plus the second edge, both with multiplicity one.
Then a calculation with {\pkg} (see Section~\ref{sec:tool}) yields
\[
GW_{0,0}^{G,\beta} = \frac{1}{t_1^2t_2 - t_1t_2^2},
\]
and
\[
\left(\mathrm{PD}(p_1)\ast_T \mathrm{PD}(p_1)\right)[q^\beta] = 
\left(\frac{t_2^2}{t_1^2 - t_1t_2}, \frac{-t_2}{t_1}, \frac{-t_2}{t_1 - t_2}, 0, 0, 0, 0, 0\right),
\]
where the vector notation indicates values under localization to the fixed points in cyclic order.
Any of those rational functions failing to be polynomial in $t_1,t_2$ proves that $G$ is not the GKM graph of a Hamiltonian or algebraic GKM space.

\begin{rem}
    In this particular case, there are multiple other ways of deducing the same result.
    First, one could prove that no moment map image of $G$ could satisfy the required convexity property: there would always be a segment on the boundary of the convex hull of the image that is not the image of a $T$-stable $S^2$ \cite[\S2.3]{GKZ20}.
    Second, one could calculate the combinatorial Betti numbers and get $(b_0,b_2,b_4)=(2, 4, 2)$, which cannot be the Betti numbers of a connected space.
    The combinatorial Betti numbers match the topological Betti numbers if the underlying GKM space is Hamiltonian \cite[Theorem~1.3.2]{GZ01}, so this gives another way of showing that $G$ is not realizable as a Hamiltonian GKM space.
\end{rem}

\begin{rem}
    As noted by \cite[Example~2.44]{GKZ22}, $G$ is realizabale for a weaker notion of GKM space where we drop the symplectic assumption and the axial function is well-defined only up to sign.
    $G$ is then realized by the equivariant connected sum of three copies of $S^2\times S^2$, each carrying the diagonal action of $T=S^1\times S^1$, with a certain choice of $T$-invariant almost complex structure.
\end{rem}

\subsubsection{Example 2: almost positive 3D GKM graphs}

The following proposition is an example of how our understanding of the local model $X_k$ (see Figure~\ref{fig:localP1}) leads to new Hamiltonian realizability criteria for GKM graphs.
Examples that satisfy the conclusion of the proposition are given by Figure~\ref{fig:example_3d_k1} ($e$ is the edge representing $\beta_1$ or $\beta_2$) and Figure~\ref{fig:F3_twisted} ($e$ is the edge representing $\beta$).
%\daniel{Is this paragraph what you asked for, Giosue?}

%\daniel{(NP) Find a good example that is not realizable by the condition below.}

\begin{prop}
\label{prop:realizability_almost_positive}
    Let $G$ be a 3-valent abstract GKM graph and $e\in E(G)$ such that $\mathcal{C}_1(e)=0$ and $[C_e]\notin\text{span}_{\mathbb{N}_{\ge0}}\{[C_{e'}]:e'\in E(G)\setminus\{e\}\}$.
    This includes the case when $G$ is almost positive (see Definition~\ref{def:almost_positive}) and $e$ is the exceptional edge.
    Let $t_1,t_2,t_3, k$ be the local parameters for $e$ as in Figure~\ref{fig:localP1}.
    
    If $G$ can be realized as compact Hamiltonian or projective algebraic GKM space then
    \begin{itemize}
        \item $k=0$, or
        \item $k=1$ and $t_3= -t_1+y t_2$ for some $y\in\mathbb{Z}\setminus\{0\}$, or
        \item $k\ge 1$ and $t_3 = -kt_1+t_2$, or 
        \item $t_3=-t_1-t_2$.
    \end{itemize}
\end{prop}
\begin{rem}
    Note that this condition on $e$ can be checked using Proposition~\ref{prop:curve_classes}.
\end{rem}
\begin{proof}
    By the assumption on $e$, all Gromov--Witten invariants in class $[C_e]$ can be computed by restriction to the local model $X_k\rightarrow C_e$.
    By Lemma~\ref{lem:polynomiality_condition}, $GW_{0,0}^{X_k,\mathbf{0}_{X_k}}$ cannot be a polynomial in the equivariant parameters if none of the conditions in the list apply.
    But if $G$ is realized by a Hamiltonian or algebraic GKM space $X$ then $GW_{0,0}^{X,[C_e]}=GW_{0,0}^{X_k,\mathbf{0}_{X_k}}$ must be polynomial.
\end{proof}

\subsection{Enumerating curves contained in a hyperplane}\label{sec:enumerating_curves}

Let us consider the following problem:
\begin{question}\label{q:planar_curves}
    Let $n,k,$ and $d$ be positive integers. How many rational curves of degree $d$ in $\mathbb{P}^n$ pass through a given set of linear subspaces in general position, under the condition that each curve lies entirely within some $k$-dimensional linear subspace?
    %\daniel{"each curve" makes me wonder: are there multiple curves? Would the following formulation also work? "[...] in general position, that are contained in some (not prescribed) $k$-dimensional linear subspace?"}\giosue{multiple curves means reducible curves, I suppose. But even with that interpretation, if I call "curve" some reducible scheme, then when I say "each curve" I mean the entire scheme with all its component. Your formulation leaves open the possibility to read that all curves must contained in the same (not prescribed) linear subspace.}
    %\daniel{okay, then let's go with your original formulation!}
\end{question}
We show how {\pkg} answers this question.
Let us denote by $S$ the flag variety $\mathrm{Fl}(1,k,n-k)$ (see Example~\ref{example:def_Fl}), i.e.,
$$S=\{(x,h)\in \mathbb{P}^n \times G(k+1, n+1) : x\in h\},$$
and by $\pi$ and $\pi'$ the projections to, respectively, $\mathbb{P}^n$ and $G(k+1,n+1)$. 
Note that curve classes in $S$ are given by combinations of classes $\beta$ and $\beta'$ where $\beta$ (resp., $\beta'$) is a line in a fiber of $\pi'$ (resp., $\pi$), see \cite[Section~9.3.1]{3264}. In particular, $\pi(\beta)$ is the class of a line and $\pi'(\beta)$ is the class of a point.

It follows that $\overline{\mathcal{M}}_{0,m}(S, d\beta)$ parametrizes stable maps to $S$ such that the projection of those maps to $\mathbb{P}^n$ is a degree $d$ curve, and the projection to the Grassmannian $G(k+1, n+1)$ is a point.
Since $S$ is homogeneous, it admits a transitive action of an algebraic group $G$. The virtual dimension of 
$\overline{\mathcal{M}}_{0,m}(S, d\beta)$ is $$(k+1)(n-k+d)+k+m-3.$$
\begin{rem}
    In \cite{MR3884168,MR4927162}, the authors consider only the case $n=3$ and $k=2$, focusing on the singularities of the curves. They construct the moduli space $\overline{\mathcal{M}}_{0,m}(E/B, \beta)$ of stable maps to $E$ relative to a fibration $E/B$. They apply this construction to the fibration $S/G(3,4)$. Finally, using WDVV equations they find a recursive formula for the number of rational \textbf{planar} curves of degree $d$ in $\mathbb{P}^3$ meeting $r$ general lines and $s$ general points, where $r+2s= 3d+2$. In this section we give an answer to Question~\ref{q:planar_curves} for any $n$ and $k$ without using WDVV equations.
\end{rem}
Let $\{\gamma_i\}_{i=1}^m$ be pure dimensional subvarieties of $\mathbb{P}^n$ such that
\begin{equation}
    \label{eq:codimension_cond}
    \sum_{i=1}^m \mathrm{codim}(\gamma_i)=(k+1)(n-k+d)+k+m-3.
\end{equation}
Using \cite[Lemma~14]{MR1492534}, there are elements $\{g_i\}_{i=1}^m\subset G$ such that scheme theoretic intersection
\begin{equation}
    \label{eq:intersectionLemma14}
    \ev_1^{-1}(g_1\cdot \pi^{-1}(\gamma_1))\cap \cdots \cap \ev_m^{-1}(g_m\cdot \pi^{-1}(\gamma_m))
\end{equation}
is a reduced number of points supported in the locus of automorphism-free maps with irreducible domain. Moreover, this number coincides with the Gromov--Witten invariant
\begin{equation}\label{eq:planar_curves}
    \int_{\overline{\mathcal{M}}_{0,m}(S, d\beta)} \ev_1^*(\pi^*(\gamma_1))\smile \cdots \smile\ev_m^*(\pi^*(\gamma_m)).
\end{equation}
The intersection in Equation~\eqref{eq:intersectionLemma14} is obtained as an application of the Kleiman--Bertini Theorem in \cite{Kle} (see for example \cite[Corollary~3.9]{MR4634985} and references therein for other examples of this well-know application), and it could be empty. 

If $f\colon \mathbb{P}^1\rightarrow S$ is a stable map representing a point of Equation~\eqref{eq:intersectionLemma14}, the image of the composition $\pi\circ f\colon\mathbb{P}^1\rightarrow \mathbb{P}^n$ is a degree $d$ curve, contained in a $k$-dimensional linear subspace, and clearly meets some translate of $\gamma_i$ for all $i$. Assuming all $\gamma_i$ in general position, we obtain that Equation~\eqref{eq:planar_curves} is the answer to Question~\ref{q:planar_curves}.

We computed some concrete examples of this equation using {\pkg} and the results were as expected. 
Moreover, we computed some new enumerative numbers of curves. In Table~\ref{tab:plane_cubics_in_P4} we list the number of cubics in $\mathbb{P}^4$ contained in a plane and meeting $a$ points, $b$ lines, and $c$ planes in general position. We impose $3a+2b+c=14$ by Equation~\eqref{eq:codimension_cond}, and $2a+b\le 6$ because we need at least one plane passing through the points and the lines. 

\begin{table}[htb]
    \centering
    \begin{tabular}{|c|c|||c|c|||c|c|||c|c|}
    \hline
        $(a,b,c)$ & $N$ & $(a,b,c)$ & $N$ & $(a,b,c)$ & $N$ & $(a,b,c)$ & $N$ \\\hline\hline
        
        $(0, 6, 2)$ & $60$ & $(3, 0, 5)$ & $12$ & $(1, 2, 7)$ & $1992$ & $(0, 2, 10)$ & $188100$ \\\hline
        
        $(1, 4, 3)$ & $24$ & $(1, 3, 5)$ & $252$ & $(2, 0, 8)$ & $1032$ & $(1, 0, 11)$ & $85500$ \\\hline
        
    $(2, 2, 4)$ & $12$ & $(2, 1, 6)$ & $144$ & $(0, 3, 8)$ & $28440$ & $(0, 1, 12)$ & $1216080$  \\\hline
    
    $(0, 5, 4)$ & $540$ & $(0, 4, 6)$ & $4080$ & $(1, 1, 9)$ & $13464$ & $(0,0,14)$  & $7833840$ \\\hline
    \end{tabular}
    \vspace{2mm}
    \caption{Number $N$ of plane cubics in $\mathbb{P}^4$ meeting $a$ points, $b$ lines, and $c$ planes.}
    \label{tab:plane_cubics_in_P4}
\end{table}

Equation~\eqref{eq:planar_curves} is zero if $d<k$. The reason, from a geometric point of view, is that in this case any irreducible rational curve of degree $d$ in $\mathbb{P}^n$ is contained in many $k$-dimensional linear subspaces. Thus, in order to have a finite number of stable maps, we need to impose some constraints on the subspaces containing the curves. More results on this problem may be found in \cite{MurPaul}.

\subsection{Smooth Schubert varieties}
\label{sec:Schubert}

In this final section, we consider smooth Schubert varieties and answer the following question given to the first author by Prof. Leonardo Mihalcea at the conference \emph{Torus Actions and Characteristic Classes} in Bedlewo, Poland, June 2025.

\begin{question}
    Does there exist a smooth Schubert variety on which some quantum product has infinitely many terms? That is, there are infinitely many curve classes $\beta$ such that the $q^\beta$ term in the quantum product does not vanish.
\end{question}

The interest in this question arises from recent advances in the development of quantum Schubert calculus for Schubert varieties (see for example \cite{LSXY25_quantum_schubert_smooth} and references therein).
While quantum Schubert calculus classically studies homogeneous spaces $X$, which are positive in the sense that $\int_\beta c_1(T_X)>0$ for every curve class $\beta\in H_2^\text{eff}(X)$, this positivity property generally fails for Schubert varieties.
Therefore, the grading of $q^\beta$ introduced in Section~\ref{sec:equivariant_QH} no longer rules out the possibility of infinitely many terms in a quantum product.

In the following, let $R$ be a root system, $G$ be the simply-connected complex Lie group with root system $R$, $B\subset G$ be a Borel subgroup, $a_1\dots,a_r$ be the simple roots defined by $B$, $S$ be a set of simple roots, and $P\subset G$ be the parabolic subgroup defined by $S$.
Let $W(G)$ be the Weyl group of $G$, generated by the simple reflections $s_1,\dots,s_r$ associated to $a_1,\dots,a_r$.
Given $w\in W(G)/W(P)$, let $X_w\subset G/P$ be the corresponding Schubert variety.
If $X_w$ is smooth, it is an algebraic GKM space with respect to the maximal torus $T\subset B$, as remarked in Example~\ref{example:Schubert_are_GKM}.

\begin{figure}[htb]
    \centering
    \includegraphics[width=0.6\linewidth]{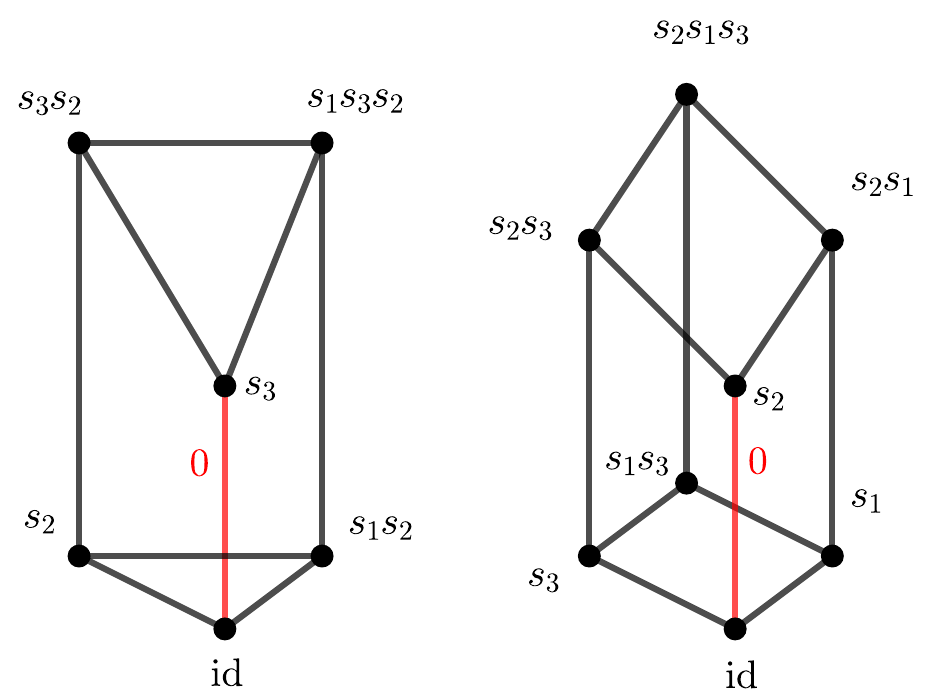} \hspace{5mm}
    \includegraphics[width=0.34\linewidth]{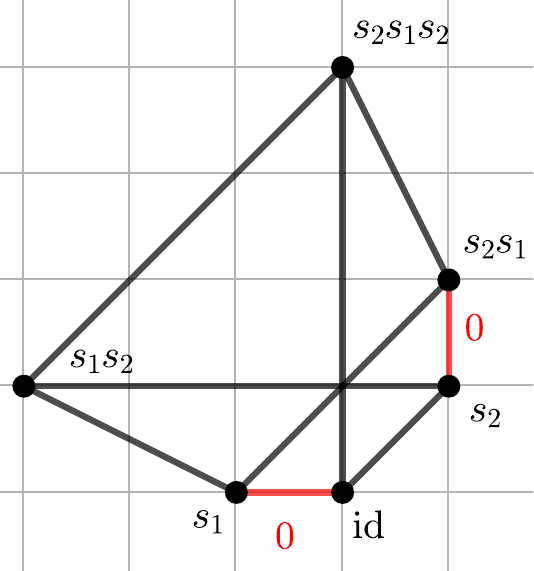}
    \\
    \vspace{3mm}
    $\begin{cases}
    R = A_3 \\
    S = \{a_1\}\\
    w = s_1s_3s_2
    \end{cases}$
    \hspace{15mm}
    $\begin{cases}
    R = A_3 \\
    S = \emptyset\\
    w = s_2s_1s_3
    \end{cases}$
    \hspace{15mm}
    $\begin{cases}
    R = G_2 \\
    S = \emptyset\\
    w = s_2s_1s_2
    \end{cases}$
    \caption{GKM graphs of some smooth Schubert varieties $X_w$ where each edge $e$ satisfies $\mathcal{C}_1(e)\ge 0$ and edges with $\mathcal{C}_1(e)=0$ are marked with a red zero.
    We omit the values of the axial function for brevity.
    From left to right, we have $X_w\subset \text{Fl}(2,1,1)$ (partial flag variety for $\mathbb{C}^4$), and $X_w\subset\text{Fl}(1,1,1,1)$ (full flag variety for $\mathbb{C}^4$), and $X_w\subset G/B$ (in type $G_2$), respectively.
    }
    \label{fig:Schubert}
\end{figure}

\begin{thm}\label{thm:Schubert}
    There exist smooth Schubert varieties $X_w$ with $x,y\in H^*(X_w)$ such that $x\ast y$ has non-zero $q^\beta$-term for infinitely many curve classes $\beta$.
\end{thm}
We prove this by considering the three examples presented in Figure~\ref{fig:Schubert}, see Propositions~\ref{prop:Schubert:A3_GP}, \ref{prop:Schubert:A3_GB}, and \ref{prop:Schubert:G2_GB} below, respectively.

\begin{prop}\label{prop:Schubert:A3_GP}
    Let $X_w\subset \mathrm{Fl}(2,1,1)$ be the smooth Schubert variety on the left of Figure~\ref{fig:Schubert}.
    Then there exist $x,y\in H^*(X_w)$ such that the non-equivariant quantum product $x\ast y$ has infinitely many terms.
\end{prop}
\begin{proof}
   Let $x,y\in H_T^2(X_w)$ be the equivariant Poincar{\'e} duals of the GKM subgraphs induced by the vertices $\{s_3s_2, s_3, s_2, \text{id}\}$ and $\{s_1s_3s_2, s_1s_2, s_3, \text{id}\}$, respectively.
   Let $\beta\in H_2(X_w)$ be the curve class defined by the edge $e:=\{\text{id},s_3\}$ in the GKM graph.
   Note that this is the unique edge with $\mathcal{C}_1(e)=0$ and all other edges have ${\mathcal{C}_1(e)>0}$.
   This edge is also the intersection of the two GKM subgraphs defining $x$ and $y$, respectively.
   We will apply Theorem~\ref{thm:twisted_Qprod} to show that $x\ast y$ has a non-zero $q^{d\beta}$-term for every $d\in\mathbb{N}_{\ge 1}$.

   Near $e$, the GKM graph of $X_w$ looks like the local model $X_k$ in Figure~\ref{fig:localP1} with $k=0$.
   Then by Theorem~\ref{thm:twisted_Qprod} we have
   \begin{align*}
   x\ast_T y = x\smile y &+ \text{PD}(C_e)
   \left(\int^T_{C_e}x\right)
   \left(\int^T_{C_e}y\right)
   \sum_{d\ge 1}q^{d\beta}\\
   &+ \left(q^\alpha\text{-terms with }\int_\alpha c_1(T_{X_w})>0\right).
   \end{align*}
    The forgetful map $H_T^*(X_w)\rightarrow H^*(Z)$ sends the equivariant Poincar{\'e} dual $PD(C_e)$ to the ordinary Poincar{\'e} dual, which is non-zero.
    It also sends $x\ast_T y$ to $x\ast y$.
    Since $\int^T_{C_e}x=\int^T_{C_e}y=-1$, the quantum products $x\ast_T y$ and $x\ast y$ both have a non-zero $q^{d\beta}$-term for each $d\in \mathbb{N}_{\ge 1}$, as desired.
\end{proof}

\begin{prop}\label{prop:Schubert:A3_GB}
    Let $X_w\subset \mathrm{Fl}(1,1,1,1)$ be the smooth Schubert variety in the middle of Figure~\ref{fig:Schubert}.
    Then there exist $x,y\in H^*(X_w)$ such that the non-equivariant quantum product $x\ast y$ has infinitely many terms.
\end{prop}
\begin{proof}
    The proof is analogous to that of Proposition~\ref{prop:Schubert:A3_GP}, except that $x,y$ are defined by the GKM subgraphs with vertices $\{\text{id}, s_2, s_2s_1, s_1\}$ and $\{\text{id}, s_2, s_2s_3, s_3\}$, respectively.
    The rest of the proof carries over word by word.
\end{proof}

Any of the preceding two examples suffice to prove Theorem~\ref{thm:Schubert}.
The reason for including the following curious example in type $G_2$ is that we are only aware of an \emph{equivariant} quantum product with infinitely many terms, but not of a non-equivariant one.

\begin{prop}\label{prop:Schubert:G2_GB}
    Let $X_w\subset G_2/B$ be the smooth Schubert variety on the right of Figure~\ref{fig:Schubert}.
    Then there exists $x\in H_T^*(X_w)$ such that the equivariant quantum product $x\ast_T x$ has infinitely many terms.
\end{prop}

\begin{proof}
    Let $x$ be the equivariant Poincar{\'e} dual of the fixed point $\text{id}\in X_w$ and let $\beta$ be the curve class of the edge $e=\{s_1,\text{id}\}$.
    This is the unique primitive curve class with $\int_\beta c_1(T_{X_w})=0$ and all other primitive curve classes $\gamma$ have $\int_\gamma c_1(T_{X_w})>0$.
    
    We show for any $d\in\mathbb{N}_{\ge 1}$ that the $q^{d\beta}$-term in $x\ast_T x$ can be calculated in terms of a small neighbourhood of $e$ on the GKM graph.
    Indeed, recalling the localization formula in Theorem~\ref{thm:LS17}, we see that any decorated tree $\overrightarrow{\Gamma}$ contributing to $x\ast_T x$ must have at least one marked vertex sent to $\text{id}$.
    But as $\overrightarrow{\Gamma}$ is connected and no edge may be sent to an edge $e'$ with $\mathcal{C}_1(e')>0$, each vertex of $\overrightarrow{\Gamma}$ must be sent to $s_1$ or to $\text{id}$.
    Hence, the argument in the proof of Theorem~\ref{thm:twisted_Qprod} still applies and shows that the $q^{d\beta}$-terms in $x\ast_T x$ can be calculated in a small neighbourhood of $e$ on the GKM graph of $X_w$.

    The neighbourhood of $e$ is is modeled by $X_k$ with $k=1$ and $t_3=t_2 -3t_1$ (see Figure~\ref{fig:localP1}).
    A calculation similar to Theorem~\ref{thm:gw_equivariant_CY} shows that
    \[
    GW_{0,0}^{X_k,d\mathbf{0}_{X_k}} = \frac{-2t_1+t_2}{d^3t_2}.
    \]
    Arguing like in Theorem~\ref{thm:twisted_Qprod} then shows that the $q^{d\beta}$ term in $x\ast_T x$ is given by
    \[
    \mathrm{PD}(C_e)\left(\int^T_{C_e}x\right) ^2 \frac{-2t_1+t_2}{t_2}.
    \]
    Since $\int_{C_e}^T x = t_2t_3$, we have
    \begin{align*}
   x\ast_T x = x\smile x &+ t_2t_3^2(-2t_1+t_2)\,\text{PD}(C_e)
   \sum_{d\ge 1}q^{d\beta}\\
   &+ \left(q^\alpha\text{-terms with }\int_\alpha c_1(T_{X_w})>0\right).
   \end{align*}
    Hence, the coefficient of each $q^{d\beta}$ in $x\ast_T x$ is non-zero, as desired.
\end{proof}
\begin{rem}
    The forgetful map $H_T^*(X_w)\rightarrow H^*(X_w)$ sends $t_1,t_2,t_3$ to $0$, so the above argument does not show that $x\ast x$ has infinitely many terms.
    Indeed, we are not aware of any non-equivariant quantum product with infinitely many terms on this Schubert variety.
\end{rem}

% \subsubsection{Further examples?}
% While all examples in Figure~\ref{fig:Schubert} have $\mathcal{C}_1(e)\ge 0$ for all edges $e$, this is not generally true for Schubert varieties.

% \daniel{We could give here a list of smooth Schubert varieties in types $A_2,A_3,B_2,B_3,C_2,C_3,G_2$ that are non-positive. It would be interesting to express the minimum Chern number (or its positivity) in terms of a combinatorial criterion on the input data (Root system, simple roots, Weyl group element).
% However, I don't know if that should rather be (NP).}
% \giosue{It seems ambitious. It could fit well in a paper dedicated to G/P-results. Why not $E_x$ and $F_x$ also?}
% \daniel{Computations for $E_x$ and $F_x$ seem to take quite long, but are of course equally interesting!}

\bibliographystyle{amsalpha}
\bibliography{refs}

\end{document}